\documentclass[sn-mathphys-num]{sn-jnl}

\usepackage{graphicx}%
\usepackage{multirow}%
\usepackage{amsmath,amssymb,amsfonts,mathtools}%
\usepackage{amsthm}%
\usepackage{mathrsfs}%
\usepackage[title]{appendix}%
\usepackage{xcolor}%
\usepackage{textcomp}%
\usepackage{manyfoot}%
\usepackage{booktabs}%
\usepackage{algorithm}%
\usepackage{algorithmicx}%
\usepackage{algpseudocode}%
\usepackage{listings}%
\usepackage{enumitem}
\usepackage{csquotes}

\allowdisplaybreaks
\def\R{\mathbb{R}}
\def\N{\mathbb{N}}

\def\Uc{\mathcal{U}}
\def\Ucs{{\mathcal{U}^*}}
\def\Vc{\mathcal{V}}

\def\Yc{\mathcal{Y}}

\def\Fc{\mathcal{F}}

\def\Ft{\prescript{}{\theta}{\mathbf{F}}}

\def\Cfu{C_{f_u}}

\def\Ldf{L_{\nabla f}}

\def\jmax{\widetilde{j}^*}

\def\setildk{e_{j}}
\def\epk{\epsilon_{K(j)}}

\def\thedag{\theta^\dagger}
\def\thej{\theta^{\delta,j}}

\def\thekj{\widetilde{\theta}^{\delta,j}}
\def\sthekt{\widetilde{\theta}^{t}}
\def\sthekj{\widetilde{\theta}^{j}}
\def\sthekjp{\widetilde{\theta}^{j-1}}

\def\sthekm{\widetilde{\theta}^{m}}
\def\sthekn{\widetilde{\theta}^{n}}

\def\thekjn{\widetilde{\theta}^{\delta,j+1}}
\def\sthekjn{\widetilde{\theta}^{j+1}}
\def\sthejsub{\widetilde{\theta}^{j_n}}

\def\thekjstar{\widetilde{\theta}^{\delta,\jmax(\delta)}}

\def\thekjsub{\widetilde{\theta}^{\delta_n,\jmax(\delta_n)}}

\def\Pitil{\widetilde{\Pi}}
\def\Pidag{{\Pi_\dagger}}
\def\util{\widetilde{u}}

\def\ztil{\widetilde{z}}
\def\Stil{\widetilde{S}_{K(j)}}
\def\sStil{\widetilde{S}_{\kappa(j)}}
\def\ydel{y^\delta}
\def\Bthe{B_R^X(\theta^\dagger)}
\def\Bu{B_r^\Vc(u^\dagger)}

\def\But{\prescript{}{\theta}{B}_{r/2}^\Vc(\ut)}
\def\sBut{\prescript{}{\theta}{B}_{r/2}^\Vc(u_*^j)}

\def\ut{u_*}
\def\usj{u_*^j}
\def\usjp{u_*^{j-1}}

\def\ukapj{u^j_{\kappa(j)}}
\def\ukapjp{u^{j-1}_{\kappa(j-1)}}

\def\MS{\overline{M}_S}

\def\MRtil{\widetilde{M}_R}

\def\({\left(}
\def\){\right)}

\newcommand{\embed}{\hookrightarrow}

\def\R{\mathbb{R}}

\def\N{\mathbb{N}}

\def\h{\textbf{h}}

\def\Om2R{\Omega,\R^3}

\newcommand{\argmin}{\mathop{\mathrm{argmin}}}

\usepackage{calc}

\usepackage{accents}
\newcommand{\dtilde}[1]{\widetilde{\raisebox{0pt}[0.8\height]{$\widetilde{#1}$}}}

\newcommand{\blue}[1]{\textcolor{black}{#1}}

\theoremstyle{thmstyleone}%
\newtheorem{theorem}{Theorem}
\newtheorem{proposition}[theorem]{Proposition}%

\theoremstyle{thmstyletwo}%
\newtheorem{remark}{Remark}%

\newtheorem{assumption}{Assumption}
\newtheorem{corollary}{Corollary}
\newtheorem{lemma}{Lemma}
\newtheorem{discussion}{Discussion}

\theoremstyle{thmstylethree}%

\begin{document}

\title[Article Title]{Sequential bi-level regularized inversion with application to hidden reaction law discovery}


\author*[1]{\fnm{Tram Thi Ngoc} \sur{Nguyen}}\email{nguyen@mps.mpg.de}

%
%
\affil*[1]{\orgdiv{Fellow Group Inverse Problems}, \orgname{Max Planck Institute for Solar Systems Research}, \orgaddress{\street{Justus-von-Liebig-Weg 3}, \city{G\"ottingen}, \postcode{37077}, \country{Germany}}}
%
%


\abstract{In this article, we develop and present a novel regularization scheme for ill-posed inverse problems governed by nonlinear \blue{time-dependent} partial differential equations (PDEs). In our recent work, we introduced a bi-level regularization framework. This study significantly improves upon the bi-level algorithm by sequentially initializing the lower-level problem, yielding accelerated convergence and demonstrable multi-scale effect, while retaining regularizing effect and allows for the usage of inexact PDE solvers. Moreover, by collecting the lower-level trajectory, we uncover an interesting connection to the incremental load method. The sequential bi-level approach illustrates its universality through several reaction-diffusion applications, in which the nonlinear reaction law needs to be determined. We moreover prove that the proposed tangential cone condition is satisfied.}




\maketitle

\section{Introduction}
This article studies nonlinear evolution equations on the general form
\begin{equation}\label{original-pde}
\begin{split}
&\dot{u}(t)+f(t,\theta,u(t))=0\qquad t\in(0,T)\\
&u(0)=A\theta,
\end{split}
\end{equation}
with the unknown $\theta$ typically appearing in the form of a source term, physical parameter, energy law etc, where the nonlinear term $f$ might represent e.g.~spatial derivatives and other nonlinear functions in $u$. Above, $\dot{u}$ denotes the time derivative of the state $u=u(t,x)$, and $u(0)$ defines its initial state, which may depend linearly on the unknown parameter $\theta$.

The aim of this work is to identify the model parameter $\theta$ in \eqref{original-pde}. To this end, we require information on the state $u$ acquired from indirect measurements
\begin{align}\label{L}
L:\Uc\to Y \qquad Lu=y.
\end{align}
In practice, the linear observation operator $L$, if not the identity, may capture only partial information on $u$, such as its value on a sub-domain, its boundary or in a snapshot. Moreover, data is often corrupted by measurement noise, so that we only have access to the noisy signal $y^\delta$ at some noise level $\delta$. 

Observation data plays a key role, giving extra information for the parameter identification process; otherwise the model \eqref{original-pde} solely is insufficient to infer $\theta$. However, the observation operator $L$ is in many cases a compact operator, compressing much of the available information on the state. This effect renders the inversion process highly challenging, due to \emph{ill-posedness} \cite[Chapter 1]{Kirsch}. The concept of ill-posedness was introduced by Hadamard in his  lecture \cite{Hadamard} in 1902, categorizing problems that violate one of three key properties: existence, uniqueness and stability.

Many techniques have been invented to regularize ill-posed inverse problem; we refer to the seminal books \cite{Tikhonov,EHNBuch,Kirsch} for a comprehensive overview on \emph{regularization theory}. Classically, the inverse parameter problem is formulated as the reduced forward map 
\begin{align}\label{G}
 G=L\circ S: X\to \Yc \qquad G(\theta)=y
\end{align}
by composing the \blue{nonlinear} parameter-to-state/solution map
\begin{align}\label{S}
S:X\to \Vc(\subseteq \Uc) \qquad S(\theta)= u \quad\text{solving PDE }\eqref{original-pde}
\end{align}
with the observation operator $L$ defined in \eqref{L}. Regularization theory suggests various techniques for stable reconstruction in this classical \emph{reduced framework}, such as Tikhonov regularization, Landweber iteration, Newton type methods and more, as well as their modifications and extensions. This work focuses on Landweber regularization; however, it does so in a novel setting that we here refer to as the \emph{bi-level framework}, particularly designed for \blue{ill-posed inverse problems} governed by nonlinear PDE models. 

The bi-level setting is strongly connected to the \emph{all-at-once framework}, which is newer than the reduced approach. The all-at-once approach was introduced in PDE-constrained optimization \cite{KunischSachs,BurgerMuehlhuberIP,HaAs01,LeHe16,orozco-ghattas-97b,shenoy-heinkenschloss-cliff-98,Taasan91} and was more recently studied in the context of inverse problems by \cite{Kaltenbacher:17,Nguyen:19}. In the all-at-once setting, rather than the reduced map \eqref{G}, one constructs the high-dimensional forward map 
\begin{align}\label{G-aao}
\begin{split}
\mathbb{G}(\theta,u):=
\begin{pmatrix}
\dot{u}+f(t,\theta,u(t))\\
u|_{t=0}-A\theta\\
Lu
\end{pmatrix}
=\begin{pmatrix}
0\\0\\y
\end{pmatrix},
\end{split}
\end{align}
optimizing the PDE residual simultaneously with the data mismatch. This formulation bypasses constructing the parameter-to-state map $S$ with \eqref{S}; equivalently, it bypasses the need for coupling the inversion implementation with a nonlinear PDE solver. 
\blue{Indeed, nonlinear solvers frequently involve additional internal iteration, such as fixed point iteration, Newton's method etc. This introduces approximation error to the primary reconstruction for the parameter \cite{Clason_2016,Kaltenbacher_2011,JinZhou,Hinze_2019}, and leads to  technicalities in the subsequent numerical analysis.}
This particular feature \blue{of bypassing the PDE solution map $S$} makes the all-at-once approach very efficient in practice; for instance, in full waveform  seismic inversion \cite{Rieder} and in magnetic particle imaging \cite{KNSW,NguyenWald:2022}. On the theoretical side, convergence and regularization analysis require milder assumptions compared to the reduced formulation; we refer to \cite{TCC21}. 

To the best of the author's knowledge, work on the bi-level framework for ill-posed inverse problems was  initiated in \cite{nguyen24}. \blue{We remark that the bi-level framework has a long history within the context of optimization and mathematical programming \cite{Bard}.} In our study, it combines the classical reduced  and the newer all-at-once setting, integrating the best features of both\blue{:
no nonlinear PDE needs to be solved exactly, but analytical properties of the solution map $S$ are nevertheless exploited}. Algorithmically, the bi-level structure consists of an upper-/outer-level for parameter identification, and embeds a lower-/inner-level for PDE solution approximation:
\begin{align*}
&\text{\{Upper-level\} iteratively approximate unknown parameter }\theta\\
& \quad \, \rotatebox[origin=c]{180}{$\Lsh$} \quad \text{\{Lower-level\} iteratively approximate state }u
\end{align*}
In the lower-level,  another forward map is constructed via
\begin{align}\label{original-pde-reform}
\Ft(u):\Vc\to \Uc^*\times H \qquad
\Ft(u):=F(\theta,u):=\begin{pmatrix}
\dot{u}+f(t,\theta,u(t))\\
u|_{t=0}-A\theta\\
\end{pmatrix}=\varphi
\end{align}
based on the PDE \eqref{original-pde}, mirroring the first two components of the all-at-once map \eqref{G-aao}. The state $u$ is obtained, not by obtaining an exact solution to the PDE, but by iteratively minimizing the PDE residual \eqref{original-pde-reform} via
\begin{align*}
u=S(\theta)\approx u_\epsilon  \qquad\text{if $\|\Ft(u_\epsilon)\| \leq \epsilon$}.
\end{align*}
This leads to significant reduction in computational costs, and is applicable to a large class of nonlinear PDEs. 
\blue{The drawback, however, is clear: the state approximation error $\epsilon$ in the lower-level causes linearization and adjoint error in the upper-level, compounding the errors already introduced by the measurement noise. Accordingly, our analysis will constitute a balancing act -- allowing $\epsilon$ to be large enough to enable computational savings, while simultaneously being small enough that the interaction between the error terms do not destroy the parameter  update in the upper-level. To this end, explicit and adaptive stopping rules for both lower- and upper-iterations must be derived.} 

With appropriate stopping rules, despite this state error and the resulting propagated errors, the bi-level algorithm guarantees convergence of the parameter reconstruction in upper-level, and has a stabilizing effect with respect to data noise. Simultaneously, it takes advantage of some properties of the solution map $S$ in the reduced setting, despite not invoking $S$ directly, leading to beneficial structural properties such as PDE well-posedness (Sections \ref{sec:bi-level}). 
Combining these two advantageous aspects into one unified framework is the new perspective of the bi-level approach \cite{nguyen24}, which is continued in this study.\\[0.5ex]

\noindent\textbf{Contribution.} This work extends the original algorithm \cite{nguyen24} in several directions. Firstly, in the standard bi-level algorithm, all the lower-iterations operate independently. Accordingly, whenever a lower-iteration is called by the upper-level, it starts from scratch, a feature shared by many other PDE-based inversion methods. In contrast, the modified Algorithm \ref{algorithm3} forces all the lower-levels to communicate by \underline{sequentially} initiating new lower-iterations at the output of the previous lower-trajectory. This sequential initialization drastically reduces the number of lower-iterates, making this an accelerated version (Theorem \ref{theo:accelerate}). 

Secondly, while \cite{nguyen24} focuses on regularization effect, this study investigates the multi-scale feature, analyzing lower-level stopping rules -- both a priori and a posteriori -- in an adaptive manner, ensuring overall convergence (Propositions \ref{prop:seq-prio}, \ref{prop:seq-poster}). This feature, which was not studied in \cite{nguyen24}, characterizes long-term error in the lower-level, and provides theoretical convergence guarantees that account for the use of inexact PDE solvers.

A somewhat surprising third observation is that if one collects the outputs of each lower-level into a single trajectory, then it becomes clear that the sequential version shares several similarities with the so-called \emph{incremental load method}, which was developed in a very different context; this comparison is expanded upon in Section \ref{sec:lowe-trajectory}. 

The fourth striking property of the  bi-level algorithms is its generality. While \cite{nguyen24} was committed to theoretical aspects, this work highlights the numerical efficiency of the bi-level scheme and its sequential adapatation. Indeed, Section \ref{sec:application} is dedicated to demonstrating the successful application of the sequential bi-level scheme to several diffusion-reaction-based applications, discovering hidden reaction laws. Each of these applications is governed by a different nonlinear evolution PDE, confirming the universal applicability of the bi-level schemes. Due to the action of the lower-level, no PDE solver is required, aligning with modern PDE solver-free concepts, e.g.~in \cite{Raissi18, Karumuri2019SimulatorfreeSO} via deep neural network. 

The article aims to provide a complete and thorough treatment of the bi-level framework, with special emphasis on demonstrating the conditions under which it is applicable and possesses convergence guarantees. We achieve this by providing detailed examination of the components of Algorithms \ref{algorithm-app}-\ref{algorithm-app-final}, and by giving examples of explicit verification of key assumptions and conditions. In particular, we prove that the tangential cone condition holds for hidden reaction law discovery problems on the form discussed above.\\[0.5ex]

\noindent\textbf{Outline. }The article is structured as follows. Section \ref{sec:bi-level} presents the bi-level algorithms. The sequential modified  version is introduced in Section \ref{sec:sequential}. Several aspects are investigated: regularization, multi-scale and acceleration effect, in addition to the analysis of the incremental lower-level trajectory. Section \ref{sec:application} is dedicated to reaction-diffusion applications including population generics (Fisher equation), enzyme kinetics (Lane-Emden equation) and  combustion (Zeldovic-Frank-Kamenetskii equation), in which hidden reaction laws will be identified. Finally, Section \ref{sec:outlook} concludes our findings and details future prospects.

\subsubsection*{Function space and notation}\label{dis-setting}
\begin{itemize}[label=,leftmargin=0cm]
\item Consider a Gelfand triple $U\embed H\embed U^*$ with Hilbert space $U$,  define $\Uc:=L^2(0,T;U)$. The state space is the Sobolev-Bocher space $\Vc:=\{u\in \Uc: \dot{u}\in \Uc^*\}$. The continuous embedding $\Vc\embed C(0,T;H)$ \cite[Section 7.2]{Roubicek} ensures well-definedness 
of the time-wise evaluation  $u(t_i),\forall i\in[0,T]$, thus of the initial time operator $A:X\to H$ in \eqref{original-pde}.

\item $X$ is the space of the unknown parameter  and $\Yc$ is the observation space. The nonlinear parameter-to-state map $S:X\to \Vc$ in \eqref{S} is assumed to be Fr\'echet differentiable. The observation map $L: \Uc\to \Yc$ in \eqref{L} is a bounded linear operator; $G=L\circ S$ then denotes the parameter-to-observation map.  

\item The PDE residual lives in $\Uc^*=L^2(0,T;U^*)$, a function space setting that is unique to the bi-level framework due to needing to optimize the PDE residual, as opposed to not appearing in the classical reduced framework.
 

\item The PDE model $f:(0,T)\times X\times U\to U^*$ is a nonlinear Carath\'eodory map. 
Under this condition, $f$ can induce the Nemytskii operator \cite{Roubicek, Troeltzsch}; by abusing notation, we thus write:
$f:X\times\Vc\to \Uc^*, [f(\theta,u)](t):=f(t,\theta,u(t))$. 

\item 
Throughout this article, the PDE \eqref{original-pde} is assumed to be solvable. We denote by $\theta^\dagger$ the exact parameter and by $u^\dagger$ the associated exact state.  $\Bthe$ is the ball in $X$ of radius $R$ around the true parameter. Similarly, $\Bu\subset\Vc$ is the ball centered at the exact state. Also, given any $\theta^j\in\Bthe$,  $u^j_*:=S(\theta^j)$ denotes the state associating to this parameter. 
\item
The superscript $\delta\geq 0$ refers to the data noise level, which is omitted when running with noise-free data. Meanwhile, the superscript $(\cdot)^j$ denotes the upper-level iteration, while the subscript $(\cdot)_k$ indicates the lower-iteration. 
\item
Generally, $D_X:X\to X^*$ denotes an isomorphism and $I_X:X^*\to X$ its inverse. $(\cdot,\cdot)_X$ is the inner product, $\langle\cdot,\cdot\rangle_{X,X^*}$ the dual paring. $A^\star$ and $A^*$ respectively denote the Banach space adjoint and Hilbert space adjoint operator.
Continuous 
Sobolev embeddings \cite{Adams,Leoni:2009,Roubicek} $X\embed Y$ will be frequently employed;  $C_{X\to Y}$ denotes their operator norms.
\end{itemize}

\section{The bi-level framework}\label{sec:bi-level}


This section summarizes the framework in \cite{nguyen24}, including Algorithms and suitable assumptions, all of which we will employ again for this article.
These stepping stones lead to the bi-level sequential version, which is the key result in Section \ref{sec:sequential}.

Algorithm \ref{algorithm2} is based on Landweber iteration, a seminal and powerful regularized gradient-based method. In the reduced setting \cite[Algorithm 1.1]{nguyen24}, the Landweber scheme corresponds to taking only steps \eqref{al:approx-theta},\eqref{al:approx-adjoint}, and requires usage of the exact PDE solution map $S$ , as opposed to approximation via the lower-level \eqref{al:approx-S}. 
In Algorithm \ref{algorithm2}, no nonlinear PDE needs to be solved exactly. Instead, its solution (state) is estimated by another Landweber scheme \eqref{al:approx-S}, iterating until an appropriate stopping index. 
This state approximation $\Stil(\thekj)$ then enters the residual, the derivative and adjoint procedure \eqref{al:approx-adjoint}, yielding approximate derivative and adjoint state. The circumflex $\widetilde{(\cdot)}$ indicates these approximations.

\begin{algorithm}[H]\caption{\textbf{Bi-level \cite[Algorithm 1.2]{nguyen24}}\hfill \it Iteration with approx. $\sStil$}\label{algorithm2}
\{Upper-level\} Initialize $\widetilde{\theta}^{\delta,0}\in\Bthe $. Update parameter:
\begin{equation}
\begin{split}\label{al:approx-theta}
& \thekjn = \thekj -\dtilde{S'}(\thekj)^*L^*\(L\Stil(\thekj)-\ydel\) \qquad j\leq \jmax(\delta),
\end{split}
\end{equation}
\{Lower-level\} where at each $j$, meaning at each $\thekj$, run:
\begin{align}\label{al:approx-S}
& u^j_0\in\sBut  \nonumber\\
& u^j_{k+1}=u^j_k-F_u'(\thekj,u^j_k)^*\(F(\thekj,u^j_k)-\varphi\) \qquad k\leq K(j) \\
& u^j_{K(j)} = :\Stil(\thekj) \nonumber
\end{align}
\{Upper-level\} Compute the approximate adjoint: 
\begin{equation}\label{al:approx-adjoint}
\begin{split}
&v:=L^*\(L\sStil(\thej)-\ydel\)\\
&\dtilde{S'}(\thekj)^*v = -\int_0^T I_Xf'_\theta(\thekj,\Stil(\thekj) )(t)^\star\ztil(t)\,dt +I_XA^\star z(0) \\[0.5ex]
&\begin{cases}
&-\dot{\ztil}(t)+f'_u(\thekj,\Stil(\thekj) )^\star \ztil(t)= D_Uv(t) \qquad  t\in(0,T)\\
&\ztil(T)=0.
\end{cases}
\end{split}
\end{equation}
\end{algorithm}



\subsubsection*{Main assumptions}\label{sec:assumptions}

Throughout the remainder of this work, we shall impose the following standing assumptions, which are adopted from \cite{nguyen24}. The first three focus on the adjoint problem, linking the lower-level to the upper-level. The next three assumptions deal with the lower-level, whose forward map  $\Ft$ is defined in \eqref{original-pde-reform}. The last assumption, meanwhile, focuses on the upper-level, where the parameter-to-observation map $G$ is the main character.

\begin{assumption}\label{summary-ass-bi}
For all $\theta$, $\hat{\theta}\in\Bthe$ and for all $u$, $\hat{u}\in\Bu$, the following all hold with some positive constants $\MS,M_S,\Cfu,\Ldf$,  $L_{F'}$, $C_{coe}$, $M_r$, $\mu_r$, $M_R$, $\mu_R$, $K_R$ and $M_r^2+\mu_r,M_R^2+\mu_R<2$.
\begin{enumerate}[label=A.\arabic*]
\item \label{summary-ass-bi-1} 
The parameter-to-state map $S:X\to\Vc$ and the parameter-to-observation map $G=L\circ S:X\to Y$ both have bounded derivatives, satisfying 
\begin{equation*}
\begin{split}
\|S'(\theta)\|_{X\to\Vc}\leq\MS,\quad \|S'(\theta)\|_{X\to\Uc}\leq  M_S,\quad \|G'(\theta)\|_{X\to Y}\leq M_R.
\end{split}
\end{equation*}
\item \label{summary-ass-bi-2} 
The solution of the adjoint equation depends continuously on the source, such that
\begin{align*}
\begin{cases}
&-\dot{z}+f'_u(\theta,u)^\star z= h \\
&z(T)=0
\end{cases} \quad\text{with}\quad \|z\|_\Uc+\|z(0)\|_H\leq \Cfu \|h\|_\Ucs.
\end{align*}
\item \label{summary-ass-bi-3} 
The model derivative is Lipschitz continuous, in the sense that
\begin{align*}
\|f'_\theta(\theta,u)-f'_\theta(\hat{\theta},\hat{u})\|_{X\to\Ucs}+\|f'_u(\theta,u)&-f'_u(\hat{\theta},\hat{u})\|_{\Uc\to\Ucs}\\&\leq \Ldf \(\|\theta-\hat{\theta}\|_X+\|u-\hat{u}\|_\Uc\).
\end{align*}
\end{enumerate}
\begin{enumerate}[label=A.\arabic*]\setcounter{enumi}{3}
\item\label{summary-ass-low-1} 
In the lower-level, the forward map $\Ft$ has bounded derivative, with
\[\|\Ft'(u)\|_{\Vc\to\Ucs\times H}\leq M_r.\]
\item\label{summary-ass-low-2}
$\Ft$ satisfies the weak tangential cone condition (wTC), that is,
\[2\(\Ft(u)-\Ft(\hat{u}),\Ft'(u)(u-\hat{u})\)_{\Ucs\times H}\geq (M_r^2+\mu_r)\|\Ft(u)-\Ft(\hat{u})\|_{\Ucs\times H}^2.\]
\item\label{summary-ass-low-3}
 $\Ft$ is coercive, satisfying
\[C_{coe}\|\Ft(u)-\Ft(u^*)\|_{\Ucs\times H}\geq \|u-u^*\|_\Vc^\alpha,\qquad\alpha\geq1.\]
\item \label{summary-ass-up-1}
In the upper-level, the strong tangential cone condition (sTC) is satisfied, with
\begin{equation}\label{cond-upper-tcc}
\begin{split}
2\(G(\theta)-G(\hat{\theta}),G'(\theta)(\theta-\hat{\theta})\)_\Yc&\geq(M_R^2+\mu_R)\|G(\theta)-G(\hat{\theta})\|_\Yc^2 \nonumber\\[0.5ex]
\|G'(\theta)(\theta-\hat{\theta})\|_\Yc&\leq K_R\|G(\theta)-G(\hat{\theta})\|_\Yc.
\end{split}
\end{equation}
\end{enumerate}
\end{assumption}

Conditions \eqref{summary-ass-bi-1}-\eqref{summary-ass-low-1} are rather standard. Coercivity \eqref{summary-ass-low-3} is inspired by the reduced setting, in which well-posedness of PDEs via coerive energy is the key \cite{Evans,Roubicek}. The tangential cone condition \eqref{summary-ass-up-1} was first introduced in the seminal work \cite{scherzer95}. We adopt this nonlinear condition here; as this condition plays a crucial role and is typically challenging to verify, we will explicitly establish it for our numerical examples in Section \ref{sec:app-tcc}.

\section{Sequential bi-level version}\label{sec:sequential}

We begin this section by making an important observation: if a small update is made in the parameter $\theta$, the change in the state $u=S(\theta)$ ought not to be large due to continuity of the parameter-to-state map. This means that the approximate states between two consecutive steps $\widetilde{\theta}^{\delta,j}$ and $\widetilde{\theta}^{\delta,j-1}$ should not be far away from each other. Obtaining the estimated state $\Stil(\thekj)$ from knowledge of $\widetilde{S}_{K(j-1)}(\widetilde{\theta}^{\delta,j-1})$ is thus considerably faster than initiating from some arbitrary $u^{j}_0$ in the neighborhood $\sBut$ like in Step \eqref{al:approx-S}. This particular piece of information leads to a key rule to leverage Algorithm \ref{algorithm2} into a new version, which forces the lower-level to sequentially inherit the outcome from the previous trajectory.

With this in mind, we propose the following modified version called sequential bi-level Algorithm \ref{algorithm3}.

\begin{algorithm}\caption{\textbf{Sequential bi-level}\hfill \it Iteration with seq. approx. $\sStil$}\label{algorithm3}
\{Upper-level\} Initialize $\widetilde{\theta}^{\delta,0}\in\Bthe $. Update parameter:
\begin{equation}
\begin{split}\label{seq-al:approx-theta}
& \thekjn = \thekj -\dtilde{S'}(\thekj)^*L^*\(L\sStil(\thekj)-\ydel\) \qquad j\leq \jmax(\delta)
\end{split}
\end{equation}
\{Lower-level\} At each $\thekj$, \underline{sequentially} initialize using previous trajectory
\begin{equation}\label{seq-al:approx-S-init}
\boxed{\,\,\, u^j_0 := u^{j-1}_{\kappa(j-1)}\,\,}
\end{equation}
\hphantom{\{Lower-level\}} and run:
\begin{equation}\label{seq-al:approx-S}
\begin{split}
& u^j_{k+1}=u^j_k-F_u'(\thekj,u^j_k)^*\(F(\thekj,u^j_k)-\varphi\) \qquad k\leq \kappa(j) \\
& u^j_{\kappa(j)} = :\sStil(\thekj)
\end{split}
\end{equation}
\{Upper-level\} Compute the approximate adjoint: 
\begin{equation}\label{seq-al:approx-adjoint}
\begin{split}
&v:=L^*\(L\sStil(\thej)-\ydel\)\\
&\dtilde{S'}(\thekj)^*v = -\int_0^T I_Xf'_\theta(\thekj,\Stil(\thekj) )(t)^\star\ztil(t)\,dt +I_XA^\star z(0) \\[0.5ex]
&\begin{cases}
&-\dot{\ztil}(t)+f'_u(\thekj,\Stil(\thekj) )^\star \ztil(t)= D_Uv(t) \qquad  t\in(0,T)\\
&\ztil(T)=0.
\end{cases}
\end{split}
\end{equation}
\end{algorithm}

\begin{remark}[Sequential bi-level]
In comparison to Algorithm \ref{algorithm2}, Algorithm \ref{algorithm3} no longer initiates the lower-levels arbitrarily within the ball $\sBut$. Instead, it employs the output of the previous lower-procedure as the initial guess in a sequential manner. As a consequence, the lower-level stopping rules in the two bi-level algorithms differ; that is, $\kappa\neq K$ in general. Indeed, we will demonstrate that this sequential initialization substantially reduces the amount of lower-iterations (Section \ref{sec:fast-effect}), while maintaining the regularization effect.
\end{remark}

\blue{
\begin{remark}[Generalization]\label{rem:generalization} Algorithm \ref{algorithm3} can be adapted in various manners:
\begin{itemize}
\item 
In the lower-level, substituting the Landweber iteration for its Nestorov version \cite{HubmerRamlau17} can be expected to yield further acceleration. More broadly, any iterative method with know rate can be employed in the lower-level, resulting in only minimal changes to the bi-level analysis, see Discussion \ref{dis:reg-multi}. On the other hand, exchanging the solution method in the upper-level has a profound impact on the complexity of the error analysis. Nevertheless, it is worthwhile to emphasize that the sequential initialization approach remains valid, and reliably yields additional acceleration.
\item
One candidate for an upper-level solver is the \emph{fractional Landweber} method introduced by Klann, Maass and Ramlau \cite{KlannMaassRamlau, Klann_2008}, e.g.~if one aims to attenuate oversmoothing artifacts, as demonstrated in \cite{HAN19,Yang21}. In the single-level setting, one minimizes the data misfit in an appropriate weighted semi-norm, computing $\argmin_\theta\frac{1}{2}\|(GG^*)^{(\beta-1)/4}(G\theta-y^\delta)\|^2$ with linear model $G$ and parameter $\beta\in[0,1]$. However, its inclusion in the bi-level framework would necessitate a thorough extension of the fractional Landweber method to nonlinear problems, which by itself already presents a highly appealing research problem. 
\item
Another option is \emph{fractional gradient descent} \cite{Wei20, Shin23}, that is, gradient descent with a fractional derivative $D^\alpha$, e.g.~a Caputo derivative of order $\alpha\in(0,1)$. This yields the update rule $\theta_{k+1}=\theta_k-\mu_k D^\alpha J(\theta)$
for a suitable cost functional $J$. The method owes much of its popularity to its applicability in machine learning, where the non-local behavior and long memory of fractional derivatives allows for enhanced performance in complex, non-convex optimization/training \cite{Elnady25}. An extension to the bi-level setting in an infinite-dimensional, regularized framework would be extremely appealing, and constitutes a major open question.
\end{itemize}
\end{remark}
}

\subsection{Regularization effect}\label{sec:reg-effect}

Returning to our study, a key property of the bi-level Algorithm \ref{algorithm2} is that it is stable when operating with noisy data, as proven in \cite{nguyen24}. There, this was demonstrated by first showing state approximation error in the lower-level via its convergence and rate, 
from which the regularization effect of the bi-level algorithm can be derived. This proof strategy immediately transfer to the sequential bi-level Algorithm \ref{algorithm3} with only minimal modification.

In this spirit, we denote by $\epsilon_j$ the state approximation error in the lower-level, that is,
\begin{align}\label{epj}
\epsilon_j:=\|\sStil(\thekj)-S(\thekj)\|_\Uc=\|u^j_{\kappa(j)}-u^j_*\|_\Uc.
\end{align}
The $\Uc$-norm, which by construction is weaker than $\Vc$-norm, suffices to carry out the analysis in the upper-level. 

\begin{theorem}[Lower-level -- convergence rate] \label{low-conv-rate}
Consider Algorithm \ref{algorithm3}. 
Assume that there is some $r>0$ such that for all $j\leq j^*$, the sequential initialisation $u^j_0:= u^{j-1}_{\kappa(j-1)}$ satisfies $u^j_0\in\sBut$. Then the following both hold:
\begin{enumerate}[label=\roman*)]
\item
Given \eqref{summary-ass-low-1}-\eqref{summary-ass-low-2},
the sequence $\{u^j_k\}_{k\in\N}$ generated by the lower-level remains in $\Bu$. Furthermore, $\{u^j_k\}$ weakly converges to the PDE solution $u^j_*$.
\item Additionally, if \eqref{summary-ass-low-3} holds, we obtain strong convergence with rates
\begin{equation}\label{low-rate}
\begin{split}
&\|u^j_k-u^j_*\|_\Vc<\(\frac{C_{coe}^2/\mu_r}{k}\)^{\frac{1}{2\alpha}} \|u^j_0-u^j_*\|^{\frac{1}{\alpha}}=\mathcal{O}\(\frac{1}{k^{1/(2\alpha)}}\),\\[0.5ex]
&\|\Ft(u^j_k)\|_{\Ucs\times H}=\mathcal{O}\(\frac{1}{k^{1/(2\alpha)}}\).
\end{split}
\end{equation}
\end{enumerate}
\end{theorem}
\begin{proof}
Algorithm \ref{algorithm3} differs from Algorithm \ref{algorithm2} only in how the lower-level is initialized \eqref{seq-al:approx-S-init}.
Algorithm \ref{algorithm2} requires that the lower-initial guess $u^j_0$ belongs to $\sBut$, the ball of radius $r/2$ around the exact state $u^j_*$. As this was explicitly assumed in the sequential case, the convergence and convergence rate obtained for the standard Algorithm \ref{algorithm2} \cite[Theorem 1]{nguyen24} applies immediately. 
\end{proof}


As with its non-sequential counterpart, the regularization effect of the sequential bi-level Algorithm \ref{algorithm3} comes from designing stopping rules in dependence with the noise level $\delta$, balancing out the interaction between the approximation error and noise level.

\begin{theorem}[Bi-level -- Regularization]\label{theo:bi-priorstop}
Consider Algorithm \ref{algorithm3} initialized at $\theta^0\in\Bthe$ with $\|\theta^0-\theta^\dagger\|\leq R_0<R$. If:
\begin{enumerate}
\item at each $j$, the lower-level iteration is terminated at the first index $\kappa(j)$ with
\begin{align}\label{condupper-lowerstop}
\epsilon_j=\mathcal{O}\(\frac{1}{\kappa(j)^{1/(2\alpha)}}\)\leq\frac{\delta}{q^j} \qquad \text{for some }q\geq1
\end{align} 
\item and the upper-level iteration is stopped at the stopping index $\jmax(\delta)$ satisfying
\begin{align}\label{reg-index}
&\delta^2\,\jmax(\delta) \(A + \frac{B}{q^{2j}}\)\leq R^2-R_0^2 \qquad\text{and}\\
\
&\delta \,\hat{\Gamma}\left(\jmax(\delta)\right) \xrightarrow{\delta\to0} 0 \qquad\hat{\Gamma}(j):=M_R\frac{Q^j-1}{Q-1}
+\frac{C}{q} \(\frac{ Q ^j-(1/q)^j}{Q -(1/q) } \) \quad\nonumber
\end{align}\label{propprior-upper-converge}
with some $A,B,C>0$ and $Q>1$ depending on constants in Assumptions \ref{summary-ass-bi}.
\end{enumerate}
Then in Algorithm \ref{algorithm3}, the iterate at the stopping index  converges to the ground truth
\begin{align*}
\thekjstar \to \thedag \qquad \delta\to 0.
\end{align*}
\end{theorem}

\begin{proof}
See \cite[Proposition 4]{nguyen24} for uniform boundedness of the parameter sequence; \cite[Proposition 5]{nguyen24} for analysis on propagation of noise, state and adjoint state errors; \cite[Theorem 2]{nguyen24} for convergence of $\{\thekjstar\}$. The constants $A$, $B$, $C$, $Q$ are detailed there.
\end{proof}

The connection between the $\hat{\Gamma}$ appearing in the above and the choice of stopping index $\jmax(\delta)$ may appear rather abstract. To give an explicit example for the case of $q=1$, one may choose
\begin{align*}
\jmax(\delta)=\min\left\{
 \frac{1}{\delta^2}\frac{R^2-R_0^2}{A+B},\,
 \ln_Q\(\frac{1}{\sqrt{\delta}}\frac{Q-1}{M_R+C}+1\)
 \right\},
\end{align*}
from which follows $\hat{\Gamma}\left(\jmax(\delta)\right) \leq 1/\sqrt{\delta}$, and \eqref{reg-index} is immediately satisfied.

\subsection{Multi-scale effect} \label{sec:multi-effect}

The preceding section has shown that given noisy measurement
, one must terminate the upper-level appropriately. It is natural to ask what happens in the noise-free case $y^\delta=y$, where one notionally could set $\jmax=\infty$, i.e.~running the upper-level infinitely. 

At first glance, setting $\delta=0$ in \eqref{condupper-lowerstop} implies $\epsilon_j=0$, meaning that one should solve the lower-problems exactly, that is, exactly solving the nonlinear PDE \eqref{original-pde}. In spite of this, one can, in fact, accept some error in the lower-level. The key is that this inexactness must be appropriately controlled in an iterative manner. We refer to this effect as the \emph{multi-scale effect} named after \cite{TrottenbergTrottenbergSchuller}, representing the first major contribution in this work.

To do so, we first need to be able to control the error in the approximate adjoint state. As $\delta=0$, it will be omitted in the notation that follows.

\begin{lemma}[Adjoint state error]\label{lem-adjointerror}
Let $\epsilon_j$ be the state approximation error at $\sthekj\in B_R(\theta^\dagger)$, as in \eqref{epj}. Then, given any other input $\theta^j\in B_R(\theta^\dagger)$, the error in the approximate adjoint state is 
\begin{align}
\label{error-adjoint}
\|\dtilde{S'}&(\theta^j)^*-S'(\sthekj)^*\|_{\Uc\to X} \nonumber\\
&\leq  \Ldf\Cfu\|I_X\|\|D_U\|\big[L_{R,r} \Cfu+1+\|A\|\Cfu\big]\((1+M_S)\|\theta^j-\sthekj\|+ \epk\) \nonumber\\
&=:C_{\nabla f,A}\((1+M_S)\|\theta^j-\sthekj\|+ \epsilon_j\).
\end{align}
\end{lemma}
\begin{proof}
 \cite[Lemma 2]{nguyen24} with $
L_{R,r}:=\Ldf(R+rC_{\Vc\to\Uc})+\|f'_\theta(\theta^\dagger,u^\dagger)\|_{X\to\Ucs}+\|f'_u(\theta^\dagger,u^\dagger)\|_{\Uc\to\Ucs}$. See \eqref{summary-ass-bi-1}-\eqref{summary-ass-bi-3} for constants $M_S$, $\Cfu$, $\Ldf$.
\end{proof}

Employing this auxiliary result allows us to analyze the effect of the lower-level precision  $\epsilon_j$ on the bi-level algorithms. In so doing, we will achieve our goal of proving the convergence of the parameter iterate sequence as the upper-level operates to infinity.


\begin{proposition}[A priori lower-stopping]\label{prop:seq-prio}
Consider the bi-level Algorithms \ref{algorithm2},  \ref{algorithm3} ran with noise-free data and initial guess $\widetilde{\theta}^0\in\Bthe$. Assume that for each $j$, the lower-level is terminated at a stopping index $\kappa(j)$, where the stopping indices are chosen so that the resulting state errors $\epsilon_j$ (c.f.~\eqref{epj}) satisfy
\begin{equation}\label{lower-prior-choice}
\begin{split}
&C_\text{pri}\sum_{j=0}^\infty\epsilon_j^2\leq R^2-R_0^2 \quad\text{with}\\& C_\text{pri}:=\|L\|^2\(\frac{2\( K_R+{C_{\nabla f,A}R}\)^2}{\mu_R-\sigma} +(1+1/\eta)\MRtil^2+{2C_{\nabla f,A}R} \)
\end{split}
\end{equation}
and $0\leq\sigma<\mu_R$; see Assumption \ref{summary-ass-bi}, \eqref{error-adjoint}, \eqref{bound-approx-adj}, \eqref{lower-prior-const} for further constants.
Then the following both hold:
\begin{enumerate}[label=\roman*)]
\item The approximate parameter sequence $\{\sthekj\}$ remains in the ball $\Bthe$.
\item There exists a subsequence $\{\sthejsub\}$ weakly converging to a solution of the inverse problem \eqref{G}. Additionally, the full image sequence converges strongly; that is, we have
\begin{align}\sthejsub \xrightharpoonup{n\to\infty} \bar{\theta},\quad G(\sthekj)\xrightarrow{j\to\infty}y \quad\text{and}\quad G(\bar{\theta})=y.
\end{align}
\end{enumerate}

\end{proposition}

\begin{remark}\label{rem:seq-prio}
The condition \eqref{lower-prior-choice} clearly shows that despite $\delta=0$, one can indeed allow state approximation error $\epsilon_j>0$ in the lower-level, as long as it adaptively decays in accordance with the upper-level operation.
\end{remark}

\begin{proof}[Proof of Proposition \ref{prop:seq-prio}]
We begin by noting that for $\sthekj\in\Bthe$, Lemma \ref{lem-adjointerror} yields boundedness of the approximate adjoint derivative in the sense that
\begin{align}\label{bound-approx-adj}
\|\dtilde{S'}(\sthekj)^*\|_{\Uc\to X}&\leq \|S'(\sthekj)\|_{\Uc\to X}+\|\dtilde{S'}(\sthekj)^*-S'(\sthekj)^*\|_{\Uc\to X} \leq M_S + C_{\nabla f,A}\epsilon_j  \nonumber 
\nonumber\\
\implies\|\dtilde{S'}(\sthekj)^*L^*\|_{\Yc\to X}&\leq \(M_S + C_{\nabla f,A}\bar{\epsilon}\)\|L\|=:\MRtil.  
\end{align}
We now evaluate the discrepancy between two consecutive upper-steps ran with clean data $y$. Recall that $\sStil(\sthekj)$ is the approximate state resulting from the lower-level approximation, and that $\sthekjn$ is the upper-iteration run with this approximate state. In particular,
\begin{align*}
&\setildk:=\|\sthekjn-\thedag\|^2-\|\sthekj-\thedag\|^2\nonumber\\
&=2\( \sthekjn-\sthekj,\sthekj-\thedag \)+\|\sthekjn-\sthekj\|^2 \nonumber\\
&= -2\(\dtilde{S'}(\sthekj)^*L^*\big( L\sStil(\sthekj)-y\big),\sthekj-\thedag\) + \|\dtilde{S'}(\sthekj)^*L^*(L\sStil(\sthekj)-y)\|^2 \nonumber\\
&=-2\(\big[\dtilde{S'}(\sthekj)^*-S'(\sthekj)^*\big] L^*\big( L\sStil(\sthekj)-y\big),\sthekj-\thedag\)\\
&\quad-2\( L\sStil(\sthekj)-y,LS'(\sthekj)(\sthekj-\thedag) \) +\|\dtilde{S'}(\sthekj)^*L^*(L\sStil(\sthekj)-y)\|^2 \nonumber\\
&\leq -2\(\big[\dtilde{S'}(\sthekj)^*-S'(\sthekj)^*\big] L^*\big( L\sStil(\sthekj)-y\big),\sthekj-\thedag\)\\
&\quad -2\( LS(\sthekj)-y,LS'(\sthekj)(\sthekj-\thedag) \)-2\( L\sStil(\sthekj)-LS(\sthekj),LS'(\sthekj)(\sthekj-\thedag) \) \nonumber\\
&\quad +(1+\eta)\|\dtilde{S'}(\thekj)^*L^*(LS(\sthekj)-y)\|^2+(1+\frac{1}{\eta})\|\dtilde{S'}(\sthekj)^*L^*(L\sStil(\sthekj)-LS(\sthekj))\|^2 \nonumber\\
&=: E_0+E_1+E_2+E_3+E_4.
\end{align*}
Above, we have inserted $LS'(\sthekj)^*$, the exact adjoint derivative at the approximate parameter -- the circumflex indicating approximation -- in the fourth estimate. In the fifth estimate, we inserted the exact forward evaluation $LS(\sthekj)$ and applied Young's inequality $(a+b)^2\leq (1+\eta)a^2+(1+1/\eta)b^2, \eta>0$ for the squared norm. 

Estimating each term individually, we obtain
\begin{align}\label{E0}
&{E_0\leq2C_{\nabla f,A} \epsilon_j\|L\|\(\|L\|\epsilon_j+\|LS(\sthekj)-y\|\)R} =: E_{01}+E_{02}  \nonumber\\
&E_1+E_2\leq -(M_R^2+\mu_R)\|LS(\sthekj)-y\|^2+2\|L\|\epsilon_j K_R\|LS(\thekj)-y\|\\
&E_3+E_4\leq  (1+\eta)\MRtil^2\|LS(\thekj)-y\|^2+(1+1/\eta)\MRtil^2\|L\|^2\epsilon_j^2 \nonumber
\end{align}
in the following manner: For $E_0$, we employ the adjoint error in Lemma \ref{lem-adjointerror}, inserting the auxiliary quantity $LS(\sthekj)$ and using that $\sthekj\in\Bthe$. For $E_1$, we extract the first part of the tangential cone condition \eqref{cond-upper-tcc}, recalling $y=LS(\thedag)$; for $E_2$, we apply the second part of \eqref{cond-upper-tcc}. For $E_3, E_4$ we make use of the boundedness of the approximate adjoint \eqref{bound-approx-adj}.

Now with 
\begin{align}\label{lower-prior-const}
  (1+\eta)\MRtil^2<M_R^2+\mu_R/2  
\end{align} we see that $E_3$ is absorbed by part of $E_1$; hence,
\begin{align*}
e_j&\leq \big(E_1+E_3\big) + \big(E_2+E_{02}\big) + \big(E_4+E_{01}\big)\\
&\leq {-\mu_R}/2 \|LS(\thekj)-y\|^2 + 2\|L\|\( K_R+{C_{\nabla f,A}R}\)\|LS(\thekj)-y\|\epsilon_j \\
&\qquad+ \|L\|^2\((1+1/\eta)\MRtil^2+{2C_{\nabla f,A}R}\)\epsilon_j^2.
\end{align*}
Developing the second term by Young's inequality $2ab \leq a^2(\mu_R-\sigma)/2 +2b^2/(\mu_R-\sigma)$, $0\leq\sigma<\mu_R$ with $a:=\|LS(\thekj)-y\|, b:=\|L\|\( K_R+{C_{\nabla f,A}R}\)\epsilon_j$, one arrives at 
\begin{align}\label{s-ej1}
e_j&=\|\sthekjn-\thedag\|^2-\|\sthekj-\thedag\|\\
&\leq {-\frac{\sigma}{2}} \|LS(\sthekj)-y\|^2+  \epsilon_j^2\(\frac{2\( K_R+{C_{\nabla f,A}R}\)^2}{\mu_R-\sigma} +(1+1/\eta)\MRtil^2+{2C_{\nabla f,A}R} \)\|L\|^2. \nonumber
\end{align}

Now, summing \eqref{s-ej1} from $j=0$ to any arbitrary $J$, one has
\begin{align*}
\|\widetilde{\theta}^{J+1}&-\thedag\|^2 + {\frac{\sigma}{2}}\sum_{j=0}^J \|LS(\sthekj)-y\|^2 \nonumber\\
&\leq \|\widetilde{\theta}^0-\thedag\|^2 + \(\frac{2\( K_R+{C_{\nabla f,A}R}\)^2}{\mu_R-\sigma} +(1+1/\eta)\MRtil^2+{2C_{\nabla f,A}R} \)\|L\|^2\sum_{j=0}^J \epsilon_j^2 \\
&\leq R_0^2 + (R^2-R_0^2)=R^2, 
\end{align*}
if the lower-level error decays fast enough as in \eqref{lower-prior-choice}, thus also
\begin{equation}\label{s-ej2}
\frac{\sigma}{2}\sum_{j=0}^\infty \| G(\sthekj) -y\|^2\leq R^2.
\end{equation}
By induction, the whole sequence $\sthekj$ remains in the ball $\Bthe$ as long as the initial guess $\widetilde{\theta}^0$ is in the ball. This uniform boundedness implies existence of a weakly convergent subsequence $\widetilde{\theta}^{j_n}$. By \eqref{s-ej2}, the whole image sequence in fact converges strongly, that is,
\begin{align}
\sthejsub \rightharpoonup \bar{\theta}
\qquad\text{and}\qquad\|G(\widetilde{\theta}^j)-y\|\to 0 \qquad\text{as }n,j\to\infty.
\end{align}
The weak tangential cone condition, i.e.~the first part of \eqref{cond-upper-tcc}, further suggests that the limit point $\bar{\theta}\in\Bthe$ is an exact parameter. Indeed,
\begin{align*}
\lim_{n\to\infty}&\frac{M_R^2+\mu_R}{2}\|G(\bar{\theta})-G(\sthejsub)\|^2\leq \lim_{n\to\infty}\( G(\bar{\theta})-G(\sthejsub),\,G'(\bar{\theta})\(\bar{\theta}-\sthejsub\)\)\\
&=\lim_{n\to\infty}\( G'(\bar{\theta})^*\(G(\bar{\theta})-y\),\,\bar{\theta}-\sthejsub\) + \lim_{n\to\infty}\( G'(\bar{\theta})^*\(y-G(\sthejsub)\),\,\bar{\theta}-\sthejsub\)\\
&=0+\lim_{n\to\infty}\( G'(\bar{\theta})^*\(y-G(\sthejsub)\),\,\bar{\theta}-\sthejsub\)\leq R M_R \lim_{n\to\infty}\|y-G(\sthejsub)\| \quad= 0.
\end{align*}
Thus, the limit point is indeed a solution to the inverse problem \eqref{G}, completing the proof.

\end{proof}

While the a priori rule in Proposition \ref{prop:seq-prio} induces convergence, it does, by design, not take advantage of posterior information. In situations where the system output $\theta\mapsto G(\theta)=L\circ S(\theta)$ can be evaluated via, e.g.~measurement, we can instead construct a posterior rule for $\epsilon_j$ that leads to an even stronger result.

\begin{proposition}[Posterior lower-stopping]\label{prop:seq-poster}
Consider the bi-level Algorithms \ref{algorithm2},  \ref{algorithm3} ran with noise-free data and initial guess $\widetilde{\theta}^0\in\Bthe$. Assume that for each $j$, the lower-level is terminated at a stopping index $\kappa(j)$, where the stopping indices are chosen so that the resulting state errors $\epsilon_j$ (c.f.~\eqref{epj}) satisfy
\begin{equation}\label{lower-poster-choice}
C_\text{pos}\epsilon_j\leq \|G(\sthekj)-y\|\quad\text{with}\quad C_\text{pos}:=\sqrt{\frac{C_\text{pri}}{{\sigma/2}-\nu}}
\end{equation}
for all $j$ with $0\leq\nu<\sigma/2$ and $C_\text{pri}$ as in \eqref{lower-prior-choice}. Then the following both hold:
\begin{enumerate}[label=\roman*)]
\item The approximate parameter sequence $\{\sthekj\}$ remains within the ball $\Bthe$.
\item $\{\sthekj\}$ converges strongly to a solution of the inverse problem \eqref{G}. Evidently, the full image sequence also converges strongly; that is, we have
\begin{align}\sthekj \xrightarrow{j\to\infty} \bar{\theta}, \quad G(\sthekj)\xrightarrow{j\to\infty}y\quad\text{and}\quad G(\bar{\theta})=y.
\end{align}
\end{enumerate}
\end{proposition}
\begin{remark}
From the above, it is clear that $\epsilon_j$ is, in fact, a zero sequence.
\end{remark}
\begin{proof}[Proof of Proposition \ref{prop:seq-poster}] The proof consists of two main steps. The first step shows weak convergence of $\{\widetilde{\theta}^{j}\}$; the second step then leverages it to achieve strong convergence.\\[1ex]
\textit{Step 1.} Our starting point is the expression \eqref{s-ej1} in Proposition \ref{prop:seq-prio}. However, in contrast to the a priori rule derived there, we now constrain the lower-level precision $\epsilon_j$ in a posterior manner, in accordance with the residual. With the posterior rule \eqref{lower-poster-choice}, the estimat \eqref{s-ej1} becomes
\begin{alignat}{3}\label{s-finitesum}
&e_j=\|\sthekjn-\thedag\|^2-\|\sthekj-\thedag\|^2 \leq {-\frac{\sigma}{2}} \|LS(\sthekj)-y\|^2+  C_\text{pri}\epsilon_j^2 \nonumber\\
&\quad\leq \left({-\frac{\sigma}{2}}+\frac{C_\text{pri}}{C_\text{pos}^2} \right)\|G(\sthekj)-y\|^2
=:-\nu \|G(\sthekj)-y\|^2<0 \nonumber\\[1ex]
&\implies\nu\sum_{j=0}^\infty \|G(\sthekj)-y\|^2< \|\widetilde{\theta}^0-\thedag\|^2\leq R^2 
\end{alignat}
showing Fej\'er monotonicity and summability of the image sequence.  While the summability \eqref{s-finitesum} is similar to that obtained in \eqref{s-ej2} of Proposition \ref{prop:seq-prio}, it is clear that Fej\'er monotonicity is stronger than the boundedness \eqref{s-ej1} derived in Proposition \ref{prop:seq-prio}. In a similar manner as before, we immediately conclude that
\begin{align*}
\exists\, \sthejsub:\quad \sthejsub \rightharpoonup \bar{\theta}\quad\text{with}\quad G(\bar{\theta})=y \qquad\text{as }n\to\infty.
\end{align*}

Moreover, Fej\'er monotonicity can suggest convergence of the whole sequence \cite{ClasonValkonen}. In fact, if $\bar{\theta}$ and $\theta^*$ are two weak accumulation points, thus also exact solutions, then the corresponding subsequences $\widetilde{\theta}^{j_n}\rightharpoonup\bar{\theta}$ and $\widetilde{\theta}^{j_m}\rightharpoonup\theta^*$ as $n,m\to\infty$ show
\begin{align}\label{samelimit}
2\|\bar{\theta}-\theta^*\|
&=\lim_{n,m\to\infty} \( \|\widetilde{\theta}^{j_n}-\theta^*\|^2-\|\widetilde{\theta}^{j_n}-\bar{\theta}\|^2 \) -  \( \|\widetilde{\theta}^{j_m}-\theta^*\|^2-\|\widetilde{\theta}^{j_m}-\bar{\theta}\|^2 \)=0.
\end{align}
These four limits exist, as the subsequences $\widetilde{\theta}^{j_n}$, $\widetilde{\theta}^{j_m}$ are extracted from the full sequence $\widetilde{\theta}^{j}$, whose distance to any true solution decreases monotonically. Passing through a sub-subsequence argument, one achieve weak convergence of the whole iterate
\begin{align*}
\widetilde{\theta}^j \rightharpoonup \bar{\theta}\quad\text{with}\quad G(\bar{\theta})=y \quad\text{when }j\to\infty,
\end{align*}
concluding the first step.\\[2ex]
\textit{Step 2.} We now leverage the previous result to obtain strong convergence of $\widetilde{\theta}^j$, taking inspiration from \cite[Theorem 2.4]{KalNeuSch08} and adapting the proof strategy therein to our approximate/bi-level algorithm. First of all, given $0\leq n<m$, we select the index
\begin{align}\label{min-residual}
t:=\argmin_{j\in[n,m]}\|G(\widetilde{\theta}^j)-y\|, \quad 0\leq n\leq t\leq  m
\end{align}
with minimum residual. We remark that although the error sequence $\|\sthekj-\thedag\|$ is strictly monotone, the same question for the image sequence has yet to be answered, unless $G$ is linear or under further assumptions (see \cite[Proposition 2]{nguyen24}).

Consider next the discrepancy
\begin{align*}
\|\sthekm-\sthekn\|^2\leq 2\|\sthekm-\sthekt\|^2+2\|\sthekn-\sthekt\|^2
\end{align*}
in the approximate sequence. We will show that the equalities
\begin{align}
\|\sthekm-\sthekt\|^2&= 2\(\sthekt-\sthekm, \sthekt -\thedag \)+\|\sthekm-\thedag\|^2-\|\sthekt-\thedag\|^2, \label{cauchy}\\
\|\sthekn-\sthekt\|^2&= 2\(\sthekt-\sthekn, \sthekt -\thedag \)+\|\sthekn-\thedag\|^2-\|\sthekt-\thedag\|^2. \label{cauchy1}
\end{align}
hold, and in so doing, will reach our objective of showing $\lim_{n\to\infty}\|\sthekm-\sthekn\|=0$. 

We prove here only the case \eqref{cauchy}; the case \eqref{cauchy1} is analogous. Observing in \eqref{cauchy} that
$\lim_{n\to\infty} \|\sthekm-\thedag\|^2-\|\sthekt-\thedag\|^2 = a-a=0$, similarly to what was done in \eqref{samelimit}, we focus on the inner product. To estimate this term, we insert the auxiliary quantities $S'(\sthekj)^*L^*$, $\sthekj$, $y$, while employing the second part of the tangential cone condition \eqref{cond-upper-tcc}, the minimum residual at $t$  \eqref{min-residual} and the posterior lower-level stopping rule \eqref{lower-poster-choice}. It now follows that
\begin{align*}
&e_{tm}:=\left|\(\sthekt-\sthekm, \sthekt -\thedag \)\right|= \left|\( \sum_{j=t}^{m-1} \dtilde{S'}(\sthekj)^*L^*\( L\sStil(\sthekj)-y\),\sthekt-\thedag \) \right|\\
&= \bigg|\sum_{j=t}^{m-1} \( L\sStil(\sthekj)-y,LS'(\sthekj)\big(\sthekt-\thedag\big) \) \\ 
& \qquad \qquad+ {\underbrace{\(\big[\dtilde{S'}(\sthekj)^*-S'(\sthekj)^*\big]L^*\( L\sStil(\sthekj)-y\),\sthekt-\thedag\)}_{\frac{1}{2}E_0 \text{ in } \eqref{E0},  \text{ notice } \|\sthekt-\thedag\|\leq R }} \bigg| \\
&\leq  \sum_{j=t}^{m-1} \bigg[ \|L\sStil(\sthekj)-LS(\sthekj)\|+ \|LS(\sthekj)-y\|\bigg] \\&\qquad\qquad\cdot\bigg[ \|LS'(\sthekj)\(\sthekt-\sthekj\)\|+\|LS'(\sthekj)\(\sthekj-\thedag\)\| \bigg]\\
&\qquad \qquad +{ \|L\|C_{\nabla f,A} \epsilon_j\(\|L\|\epsilon_j+\|LS(\sthekj)-y\|\)R}\\
&\leq  \sum_{j=t}^{m-1} \bigg[ \|L\|\epsilon_j+\|G(\sthekj)-y\|\bigg] \bigg[K_R\|G(\sthekt)-G(\sthekj)\|+K_R\|G(\sthekj)-y\|\bigg]\\
&\qquad \qquad +{ \|L\|C_{\nabla f,A} \epsilon_j\(\|L\|\epsilon_j+\|LS(\sthekj)-y\|\)R}\\
&\leq  \sum_{j=t}^{m-1} \bigg[ \|L\|\epsilon_j+\|G(\sthekj)-y\|\bigg] \bigg[K_R\|G(\sthekt)-y)\|+K_R\|G(\sthekj)-y\|+K_R\|G(\sthekj)-y\|\bigg]\\
&\qquad \qquad +{ \|L\|C_{\nabla f,A} \epsilon_j\(\|L\|\epsilon_j+\|G(\sthekj)-y\|\)R}\\
&\leq  \sum_{j=t}^{m-1} \bigg[ \frac{1}{C_\text{pos}}\|L\|\|G(\sthekj)-y\|+\|G(\sthekj)-y\|\bigg] \bigg[3K_R\|G(\sthekj)-y\| \bigg]\\
&\qquad \qquad +{ \|L\|RC_{\nabla f,A} \frac{1}{C_\text{pos}}\(\|L\|\frac{1}{C_\text{pos}}+1\)\|G(\sthekj)-y\|^2}\\
&\leq  \(3K_R+\frac{\|L\|RC_{\nabla f,A} }{C_\text{pos}} \)\(\frac{\|L\|}{C_\text{pos}}+1\) \sum_{j=t}^{m-1} \|G(\sthekj)-y\|^2 \xrightarrow{m\to\infty}0
\end{align*}
using finite image summation \eqref{s-finitesum}. We deduce that the approximate iterate $\sthekj$ is a Cauchy sequence, obtaining strong convergence
\begin{align*}
\widetilde{\theta}^j \to \bar{\theta}\quad\text{with}\quad G(\bar{\theta})=y \quad\text{as} \quad j\to\infty.
\end{align*}
Step 2 completes the proof.
\end{proof}

We close this section with some remarks.
\begin{remark}[A priori vs.~posterior]\,
\begin{itemize}
\item 
The a priori stopping rule \eqref{lower-prior-choice}  yields weaker results than the posterior rule \eqref{lower-poster-choice}.

\item
Summing over all indices of the posterior rule and  using \eqref{s-finitesum} in fact suggests a prior stopping rule
\begin{align*}
&C_\text{pos}^2\sum_{j=0}^\infty\epsilon_j^2\leq \sum_{j=0}^\infty \|G(\sthekj)-y\|^2\leq \frac{R^2}{\nu}.
\end{align*}
The posterior rule is more restrictive than this \enquote{summed prior rule}, as the former constrains the lower-level precision at every upper-step rather than their sum.\\
\end{itemize}
\end{remark}

\begin{discussion}[Regularization vs.~multi-scale]\,\label{dis:reg-multi}
\begin{itemize}
\item 
Regularization in Section \ref{sec:reg-effect} studies convergence of $\thekjsub$, the parameter sequence at the stopping index $\jmax(\delta)$, when noise decays $\delta_n\to0$. The multi-scale effect, on the other hand, investigates the operation with clean data, quantifying the imprecision allowed in the lower-level, and proposing an iterative refinement strategies for convergence of $\sthekj$ as $j\to\infty$.

\item 
If one embeds any discretized PDE solver (finite difference, FEM, multigrid etc.) into the lower-level, our analysis informs an adaptive discretization strategy, gradually refining multigrid size, mesh size or resolution. By not fixing a single fine scale, computational efficiency is improved, while retaining theoretical guarantees for stable reconstruction;
\blue{we refer to \cite{AarsetNguyen} for an example in aerocoustics with the Helmholtz equation.}

\end{itemize}
\end{discussion}

\subsection{Acceleration effect}\label{sec:fast-effect}
This section will show that by sequential initialization, Algorithm \ref{algorithm3} can drastically reduce the amount of lower-iterations. For this, we recall that $K(j)$ denotes the stopping index of the $j$-th lower-level in the non-sequential Algorithm \ref{algorithm2} , while $\kappa(j)$ denotes the stopping index of the $j$-th lower-level in the sequential Algorithm \ref{algorithm3} in the noise free case. 

The fundamental observation we here make is that Algorithm \ref{algorithm2} starts the lower-iteration arbitrarily in $\sBut$. As such, the initial guess error stays constant $r/2$ in all lower-levels; no further information is gained as the algorithm progresses. Algorithm \ref{algorithm3} with sequential initialization can, in fact, provide a significantly better prior.



\begin{theorem}[Acceleration]\label{theo:accelerate}
Given the setting in Propositions \ref{prop:seq-prio} or \ref{prop:seq-poster}, and that $\alpha=1$ in \eqref{summary-ass-low-3}. 
Then Algorithm \ref{algorithm3} requires fewer iterations in each lower-level than that of Algorithm \ref{algorithm2}, in the sense that
\begin{align*}
\kappa(j) = o\(1/\epsilon_j^2\),  \qquad K(j) = \mathcal{O}\(1/\epsilon_j^2\).
\end{align*}
\begin{proof}
Our strategy is to examine the initial guess errors \[r_j^2:=\|u_0^j-\usj\|^2_\Vc\]  to show that they forms a zero sequence, while taking advantage of the improved initialization $u_0^j=\ukapjp$. By inserting the auxiliary quantity $\usjp=S(\sthekjp)$ with Young's inequality $ab\leq \lambda a+b/(4\lambda), \lambda>0$, then employing the convergence rate \eqref{low-rate} in the lower-level with $\alpha=1$, we obtain
\begin{align*}
r_j^2   & 
=\|\ukapjp-\usj\|_\Vc^2
\leq \lambda\|\ukapjp-\usjp\|^2 +\frac{1}{4\lambda}\|\usjp-\usj\|^2 \\
 &< \frac{C_{coe}^2/\mu_r}{\kappa(j-1)}\lambda\|u^{j-1}_0-\usjp\|^2 +\frac{1}{4\lambda}\|S(\sthekjp)-S(\sthekj)\|^2\\
& = \frac{C_{coe}^2/\mu_r}{\kappa(j-1)}\lambda r_{j-1}^2 +\frac{1}{4\lambda}\|S(\sthekjp)-S(\sthekj)\|^2,
\end{align*}
where $\usjp$, $\usj$ are the exact states associated with the parameters $\sthekj$, $\sthekjn$. Evaluating the second term above by way of the mean value theorem, one has
\begin{align*}
\|S(\sthekjp)-S(\sthekj)\|_\Vc&=\left\|\int_0^1 S'(\sthekjp+\lambda(\sthekj-\sthekjp))\,d\lambda(\sthekj-\sthekjp) \right\|
\leq \MS\|\sthekj-\sthekjp\|_X
\end{align*}
due to bounded derivative assumption \eqref{summary-ass-bi-1} along with the fact that $\sthekjp$, $\sthekj$ remain in the ball $\Bthe$ as shown in Propositions \ref{prop:seq-prio}-\ref{prop:seq-poster}, i). Now follows the recursive relation
\begin{align}\label{rj1}
r_j^2&<\frac{C_{coe}^2/\mu_r}{\kappa(j-1)} \lambda r_{j-1}^2 +\frac{\MS^2}{4\lambda}\|\sthekj-\sthekjp \|^2 _X\nonumber\\
& =\frac{C_{coe}^2/\mu_r}{\kappa(j-1)} \lambda r_{j-1}^2 +\frac{\MS^2}{4\lambda}\| \dtilde{S'}(\sthekjp)^*L^*\(L{\sStil}(\sthekjp)-y\|^2 \)\nonumber\\
&\leq\frac{C_{coe}^2/\mu_r}{\kappa(j-1)} \lambda r_{j-1}^2\nonumber\\ & \qquad+\frac{\MS^2}{2\lambda}\| \dtilde{S'}(\sthekjp)^*L^*\|^2 \(\|LS(\sthekjp)-y\|^2+\|LS(\sthekjp)-L{\sStil}(\sthekjp)\|^2\) \nonumber\\
&\leq \frac{C_{coe}^2}{\mu_r} \lambda r_{j-1}^2 +\frac{\MS^2\MRtil^2}{2\lambda}\(\|G(\sthekjp)-y\|^2 + {\|L\|^2\epsilon_{j-1}^2}\)
\end{align}
with $\MRtil$ as in \eqref{bound-approx-adj} and state error $\epsilon_{j-1}=\|S(\sthekjp)-{\sStil}(\sthekjp)\|_\Uc$. Note that $\kappa(j-1)\geq 1$, as at least one iteration takes place in each lower-level.

As the sequence $r_j$ is assumed to be bounded by $r$, taking limit superior of both sides of \eqref{rj1} with setting $\lambda=\mu_r/(2C_{coe}^2)$ yields
\begin{align*}
0\leq\limsup_{j\to\infty}r_j^2<\frac{1}{2}\limsup_{j\to\infty}r^2_{j-1}+C\lim_{j\to\infty}\(\|G(\sthekjp)-y\|^2 +{\|L\|^2\epsilon_{j-1}^2}\)=\frac{1}{2}\limsup_{j\to\infty}r_{j-1}^2
\end{align*}
using the property in Propositions \ref{prop:seq-prio}, \ref{prop:seq-poster} that the image and lower-precision sequences are zero sequences. We remind ourselves that invoking the a priori or posterior rules in these propositions results in sub-sequential or full convergence of the parameter sequence, respectively, but that they in both cases lead to convergence of the full image series. This implies
\begin{align}\label{rj-lim}
\lim_{j\to\infty}r_j^2=0,
\end{align}
   The limit \eqref{rj-lim} reveals an important fact: in Algorithm \ref{algorithm3}, the lower-level radii/prior error sequence $r_j$ does not stay constant as in Algorithm \ref{algorithm2}, but rather decays over time.

In the last step, note that the lower-level stopping indices $\kappa(j)$ satisfy
\begin{align*}
\epsilon_j^2=\|\ukapj-\usj\|_\Uc^2\leq C_{\Vc\to\Uc}\frac{C_{coe}^2/\mu_r}{\kappa(j)}  \|u_0^{j}-\usj\|^2_\Vc=: C\frac{r_j^2}{\kappa(j)}.
\end{align*}
In Algorithm \ref{algorithm2}, as $r_j=r/2$ for all $j$, $K(j)=\mathcal{O}\big(1/\epsilon_j^2\big)$. In Algorithm \ref{algorithm3}, as $\lim_{j\to\infty}r_j^2=0$, one assets $\kappa(j)=o\big(1/\epsilon_j^2\big)$. The proof is then complete.
\end{proof}
\end{theorem}


\begin{corollary}\label{cor:radius}
In Theorem \ref{theo:accelerate}, the lower-level radii is uniformly bounded by
\begin{align}\label{rj-bound}
    r_j^2\leq r_0^2+C_R,
\end{align}
with $C_R$ depending on $R$ detailed in the proof.
\end{corollary}
\begin{proof}
Indeed,  choose $\lambda= \mu_r/C_{coe}^2$ and summing \eqref{rj1} up to any arbitrary index $j$  yield
\begin{align*}
r_j^2-r^2_0&=\sum_{i=1}^j r_i^2-r_{i-1}^2< \frac{C_{coe}^2\MS^2{\MRtil^2}}{{2\mu_r}}\sum_{i=1}^j \(\|G(\sthekjp)-y\|^2 +{\|L\|^2\epsilon_{j-1}^2}\) \nonumber\\
&\leq C_R:=
\begin{cases}
\dfrac{C_{coe}^2\MS^2{\MRtil^2}}{2\mu_r}\(\dfrac{2R^2}{\sigma}+ \dfrac{\|L\|^2(R^2-R_0^2)}{C_\text{pri}} \)\\[0.5ex]
\dfrac{C_{coe}^2\MS^2{\MRtil^2}}{{2\mu_r}}\(\dfrac{R^2}{\nu}+\dfrac{\|L\|^2R^2}{\nu C_\text{pos}^2}\)
\end{cases} \nonumber
\end{align*}
depending on whether the a priori (Proposition \ref{prop:seq-prio}) or posterior stopping rule (Proposition \ref{prop:seq-poster}) is imposed in the lower-levels. 
\end{proof}

\begin{discussion}[Acceleration effect]\label{dis:improveAl3}\,
\begin{itemize}
\item
Theorem \ref{theo:accelerate} informs the limiting behavior the bi-level schemes. In practice, we can never run any algorithm infinitely due to machine precision. By fixing a single precision in all lower-levels, one can observe the acceleration effect.
\item
Even if only a single iteration takes place in each lower-level with the sequential algorithm, the state error $\epsilon_j$ nevertheless shows a decaying tendency, in practice outperforming the non-sequential algorithm. Section \ref{sec:numerics}, Fig.~\ref{fig-biseq}-\ref{fig-biseq-2} show various experiments demonstrating this effect.

\end{itemize}
\end{discussion}

\begin{discussion}[Initialization impact]\,
\begin{itemize}
\item 
Theorem \ref{theo:accelerate} reflects the important role of a prior/initial guess in nonlinear inversion, which is often chosen empirically. Algorithm \ref{algorithm3} proposes an explicit rule via sequential initialization, whose positive impact on the whole process has been analyzed.
\item 
In the non-sequential Algorithm \ref{algorithm2}, the lower-radius $r$ of $\But$ and upper-radius $R$ of $\Bthe$ are independent objects (Section \ref{sec:assumptions}). In the sequential Algorithm \ref{algorithm3}, where the lower-radii form a sequence $r_j$, it is meanwhile clear from \eqref{rj-bound} that they depend closely on the choice of $R$ and $r_0$.
\end{itemize}
\end{discussion}


\subsection{Incremental lower-level trajectories}\label{sec:lowe-trajectory}

The primary purpose of each lower-level is to approximate the state at an approximate parameter $\thekj$. This will generally differ from the true state $u^\dagger$ at the true parameter $\thedag$; that is, \[u^j_{\kappa(j)} = \sStil(\thekj)\approx S(\thekj)\quad\neq\quad S(\thedag)=u^\dagger\] even if $\delta=0$. It is, however, interesting to observe that if we collect the outputs $u^j_{\kappa(j)}$ of all lower-level trajectories, they can, in fact, approximate the true state $u^\dagger$. We begin by demonstrating this effect under the a prior stopping rule Proposition \ref{prop:seq-prio}.

\begin{lemma}[A priori approximate $u^\dagger$]\label{lem:approx-u-prior}

Consider the setting in Proposition \ref{prop:seq-prio}, and assume that the parameter-to-state map $S$ is weak-to-weak continuous. The output of the lower-level trajectories then has a subsequence $u^{j_n}_{\kappa(j_n)}$  weakly converging to the exact state $u^\dagger$, in the sense that
\begin{align}
\widetilde{S}_{\kappa(j_n)}(\sthejsub)=u^{j_n}_{\kappa(j_n)}\quad \xrightharpoonup[n\to\infty]{\Uc}  \quad u^\dagger=S(\thedag).
\end{align} 
\end{lemma}

\begin{proof} 
Proposition \ref{prop:seq-prio} has shown existence of a weakly convergent subsequence $\sthejsub$ to $\thedag$, hence
\begin{align*}
 u^{j_n}_{\kappa(j_n)}-u^\dagger=  \(\widetilde{S}_{\kappa(j_n)}(\sthejsub)-S(\sthejsub)\) +\( S(\sthejsub)-S(\thedag)\)\xrightarrow{n\to\infty}0.
\end{align*}
as the state approximation error is a zero sequence and $S$ is weakly continuous.
\end{proof}

A similar effect also holds under the posterior stopping rule Proposition \ref{prop:seq-poster}; it even holds without the need for further assumptions on the parameter-to-state map.

\begin{lemma}[Posterior approximate $u^\dagger$]\label{lem:approx-u-poster} 

Consider the setting in Proposition \ref{prop:seq-poster}. The output of the lower-level trajectories $u^j_{\kappa(j)}$ then converges to $u^\dagger$, in the sense that
\begin{align}
\sStil(\sthekj)=u^j_{\kappa(j)}\quad  \xrightarrow[j\to\infty]{\Uc}\quad u^\dagger=S(\thedag).
\end{align}
\end{lemma}

\begin{proof}
From the posterior rule \eqref{lower-poster-choice} and boundedness of the derivative \eqref{summary-ass-bi-1} , we deduce
\begin{align*}
\|u^j_{\kappa(j)}-u^\dagger\|_\Uc&\leq \|\sStil(\sthekj)-S(\sthekj)\|_\Uc+\|S(\sthekj)-S(\thedag)\|_\Uc\\
&\leq\epsilon_j+\sup_{\theta\in\Bthe}\|S'(\theta)\|_{X\to\Uc}\|\sthekj-\thedag\|_X\leq\epsilon_j+M_S\|\sthekj-\thedag\|\xrightarrow{j\to\infty}0,
\end{align*}
where  convergence of $\epsilon_j$ and of the parameter have been proven in Proposition \ref{prop:seq-poster}.
\end{proof}

\begin{remark}[Approximation in $\Vc$-norm]
A stronger approximation property can be achieved by imposing  $\epsilon_j^\Vc:=\|u_*-u^j_{\kappa(j)}\|_\Vc$ in the posterior rule \eqref{lower-poster-choice}. Indeed, in the proof of Proposition \ref{prop:seq-poster}, step \eqref{s-finitesum}, we bound $\epsilon_j$ above by $\epsilon_j^\Vc$ as $\Vc\embed\Uc$; the rest of the proof remains unchanged. Consequently,
\begin{align*}
\|u^j_{\kappa(j)}-u^\dagger\|_\Vc\leq\epsilon^\Vc_j+\MS\|\sthekj-\thedag\|\quad \xrightarrow{j\to\infty}0.
\end{align*}
with $\MS$ as in Assumption \eqref{summary-ass-bi-1}.
\end{remark}

Furthermore, noisy iterates also posses this approximation property. In this case, we are not looking at $\{\sStil(\sthekj)\}_{j\in\N}$, rather the sequence at the stopping index $\jmax(\delta_n)$, meaning $\{\widetilde{S}_{\kappa(\jmax(\delta_n))}(\thekjsub)\}_{n\in\N}$ as regularization suggests (see Theorem \ref{theo:bi-priorstop}, Discussion \ref{dis:reg-multi}). 

\begin{lemma}[Regularized approximate $u^\dagger$]\label{lem:approx-u-reg}

Consider the setting in Theorem \ref{theo:bi-priorstop}. Then the regularized state sequence at the stopping index $\jmax(\delta_n)$ converges to the exact state $u^\dagger$, in the sense that
\begin{align}
\widetilde{S}_{\kappa(\jmax(\delta_n))}(\thekjsub)=u^{\jmax(\delta_n)}_{\kappa(\jmax(\delta_n))}(\thekjsub)\quad  \xrightarrow{\,\,\, \Uc\,\,\,}\quad u^\dagger=S(\thedag)
\end{align} 
when the noise level $\delta_n$ tends to zero as $n\to\infty$.
\end{lemma}

\begin{proof}
Similarly to Lemma \ref{lem:approx-u-poster}, we have the decomposition
\begin{align*}
&\|u^{\jmax(\delta_n)}_{\kappa(\jmax(\delta_n))}(\thekjsub)-u^\dagger\|\\
&\qquad\leq \|\widetilde{S}_{\kappa(\jmax(\delta_n))}(\thekjsub)-S(\thekjsub)\|+\|S(\thekjsub)-S(\thedag)\|\\
&\qquad\leq\epsilon_{\,\jmax(\delta_n)}+M_S\|\thekjsub-\thedag\| \frac{\delta_n}{q^{\jmax(\delta_n)}}+M_S\|\thekjsub-\thedag\|\xrightarrow{n\to\infty}0
\end{align*}
where $\epsilon_{\,\jmax(\delta_n)}\leq \delta_n/q^{\jmax(\delta_n)}, q\geq 1$ is the stopping rule \eqref{condupper-lowerstop}, and parameter convergence is a consequence of the regularization effect Theorem \ref{theo:bi-priorstop}.
\end{proof}

We conclude this section with an interesting observation.

\begin{discussion}[Connection to the incremental load method (ILM)]\label{dis:incremental}
In nonlinear elasticity, an elastic body undergoes large deformations caused by external load, and the stress-strain mechanism is governed by a highly nonlinear PDE \cite{Zeidler4}. 
Finding the displacement field based on standard iterative PDE solvers exhibits extreme instability \cite{Liu24}; a more robust alternative is the ILM \cite{Ciarlet,Lorenzo22}. ILM designs a loading path that incrementally reaches the final configuration,  passing through a sequence of intermediate states. When one intermediate state is found, it is used as an initial seed for the next evaluation.

This leads to an interesting comparison to the bi-level scheme: the lower-level approximately computes the intermediate states $u^j_{\kappa(j)}$, while $\thekj$  plays the role of the load. Lemmas \ref{lem:approx-u-poster}-\ref{lem:approx-u-reg} have shown convergence of $u^j_{\kappa(j)}$ to the exact state $u^\dagger$, when the load increments to the true configuration $\thedag$. In addition, the sequential Algorithm \ref{algorithm3} uses the intermediate state $u^j_{\kappa(j)}$ as the starting point to iterate to the next state, a special feature of ILM. We demonstrate this link in Section \ref{sec:numerics}, Fig.~\ref{fig-traject-U}.
\end{discussion}

\section{Application to hidden reaction law discovery}\label{sec:application}
\subsection{Nonlinear reaction-diffusion}
We now demonstrate the practical performance of the bi-level algorithms in recovering various reaction laws in microscopic nonlinear diffusion models. In particular, we investigate the evolution diffusion-reaction equation 
\begin{equation}\label{app-eq}
\begin{split}
\dot{u} -\nabla\cdot(a\nabla u)+bu+\Pi(u)&=\phi \qquad (t,x)\in(0,T)\times\Omega\\
u(t=0)&=u_0 \qquad x\in\Omega
\end{split}
\end{equation}
with zero boundary conditions and unknown nonlinear reaction $\Pi$. Here, $\Omega$ is a bounded, Lipschitz domain in $\R^d$, $d=1,2,3$ with Lipschitz boundary; the coefficients $a$ and $c$ are spatially dependent. The reaction $\Pi$ is determined from observations $y$ in the form of full or final time measurement of the state $u$, that is, 
\begin{align}
y=Lu=u \qquad\text{or}\qquad y=Lu=u|_{t=T}.
\end{align}

Reaction-diffusion equations arise naturally in many applications and systems consisting of interacting components, such as chemical reactions \cite{Pao}. Equations on the form \eqref{app-eq} are widely used to describe various pattern-forming phenomena in physics, chemistry and biology, with relevant choices of $\Pi(u)$ including:
\begin{itemize}[label=$\circ$]
\item $u(1 - u)$: Fisher equation in population genetics \cite{Fisher}
\item $u (1 - u) (u -a), a\in(0,1)$: Allen-Cahn equation in natural convection, \blue{phase separation} \cite{AllenCahn,GildingKersner}
\item $-u|u|^p, p\geq 1$: Ginzburg–Landau equation in radiation, superconductivity \cite{BronsardStoh, HoffmannTang}
\item$u(1-u)e^{-a(1-u)}, a\in(0,\infty)$: Zeldovic-Frank-Kamenetskii equation in combustion theory \cite{FrankKamenetskii}
\item $-u/(1+au+bu^2), a^2<4b$: Lane-Emden equation in enzyme kinetics \cite{Pao}, \blue{composite stellar models \cite{Kippenhahn}}
\end{itemize}
Interestingly, reaction-diffusion models also feature in modern signal processing applications, such as segmenting, inpainting and deblurring, an approach that has yielded great success \cite{BENES04,LiInpainting01,Carasso}. 

From this, it is apparent that equations on the general form \eqref{app-eq} are a flexible tool appearing in several different disciplines, and that the nonlinear reaction $\Pi$ plays the central role, leading to distinct and interesting properties for differing use cases. This perspective justifies our interest in \emph{model discovery}: If it is known that a system is governed by \eqref{app-eq}, but $\Pi$ is either unknown or known only approximately, then recovery of $\Pi$ is likely to yield an improved treatment of the underlying problem.

\subsubsection*{Functional framework}
We cast the inverse reaction problem in the framework
\begin{alignat}{3}\label{app-setting}
&u\in\Vc=L^2(0,T;U)\cap H^1(0,T;U^*),\qquad &&U:=H_0^1(\Omega) \nonumber\\
&\phi\in \Uc^*= L^2(0,T;U^*), &&u_0\in H=L^2(\Omega),\\
&y\in \Yc:=L^2(0,T;Y) && Y:= L^2(\Omega) \nonumber
\end{alignat}
with diffusivity $0<a\in L^\infty(\Omega)$, potential $b\in L^2(\Omega)$ and zero boundaries incorporated into $H^1_0(\Omega)$. We also note that the integration-by-parts in Bochner spaces \cite[Lemma 7.3]{Roubicek}
\[\int_0^T \langle \dot{u},v\rangle_{U^*,U}+\langle u,\dot{v}\rangle_{U,U^*}\,dt=\(u(T),v(T)\)_H-\(u(0),v(0)\)_H.\]
 will be employed in adjoint derivation later. 
 
In this setting, the unknown nonlinear reaction $\Pi$ is a Nemytskii superposition operator 
$\Pi:\Vc\to\Uc^*,[\Pi(u)](t,x):=\Pi(u(t,x))$
induced by the Carath\'eodory map
$\Pi\in C(\R)$. To set up the problem in the Hilbert space framework, we choose the parameter space $X$ as
\begin{align}\label{app-setting1}
&\Pi\in X:=H^s(\R)\embed W^{1,\infty}(\R) \quad s>1, \qquad \Pidag\in W^{2,\infty}(\R).
\end{align}
As clearly $H^s(\R)\embed C(\R)$ for $s>1/2$, this suffices for well-definedness of $\Pi$ via its Carath\'eodory generator. However, we will later also consider $s>1$, as this allows for differentiability and convergence guarantee via the tangential cone condition. The ground truth $\Pidag$ is assumed to have higher regularity; its role becomes visible in verification of the tangential cone condition in Section \ref{sec:app-tcc}. 

With our problem and setting established, we next describe the different elements appearing in the bi-level algorithms.


\subsection{Bi-level adjoint structure}
\subsubsection*{Upper-level adjoint -- full measurement}
For full measurement $L=\text{Id}:\Uc\to\Yc$, the parameter-to-observation $G$ boils down to $G=S:X\to\Uc$. The derivative of 
$S:X\ni\Pi\mapsto u\in\Vc$ 
at $\Pi$ in the direction $\xi$ is 
\begin{equation}\label{app-linearized-eq}
S'(\Pi)\xi=:p \quad\text{with}\quad
\Bigg\{ \begin{array}{r@{\hspace{1ex}}c@{\hspace{1ex}}l}
\dot{p} -\nabla\cdot(a\nabla p)+bp+\Pi'(u)p & = & -\xi(u) \\
p(0) & = & 0
\end{array}
\end{equation}
where $u=S(\Pi)$ is the solution to the nonlinear equation \eqref{app-eq} with reaction term $\Pi$.

The adjoint equation in \eqref{seq-al:approx-adjoint} with input residual $v\in \Yc$ is the backward heat equation
\begin{align}\label{app-adjoint-eq}
\Bigg\{ \begin{array}{r@{\hspace{1ex}}c@{\hspace{1ex}}l}
-\dot{z} -\nabla\cdot(a\nabla z)+bz+\Pi'(u)z & = & v \qquad t\in(0,T)\\
z(T) & = & 0 
\end{array}
\end{align}
using integration-by-parts and the Banach space adjoint $\langle \Pi'(u)h,z\rangle=\langle h,\Pi'(u)z\rangle$. Compared to \eqref{seq-al:approx-adjoint}, the isomorphism $D_U$ in the source is skipped, as we already incorporated the observation space $\Yc=L^2(0,T;L^2(\Omega))$, simplifying $L^*L$ when $L=\text{Id}$.

For the adjoint integral, since the parameter and state are defined on different domains ($\R$ and $\Omega$), we are not deriving $S'(\Pi)^*v$ via the isomorphism  $I_U$, but rather directly, 
using the Fourier transform $\Fc$, employing the identity $( \xi,h )_X = \int_\R (1+|\omega|^2)^s \Fc(\xi)(\omega)\overline{\Fc(h)}\omega)\,d\omega$. Paring the right hand side of the linearized equation \eqref{app-linearized-eq} with $z$ in \eqref{app-adjoint-eq}, one has
\begin{align}\label{app-adjoint-int}
 &\langle -\xi(u),z \rangle_{\Uc^*,\Uc} \\&\quad=	-\int_0^T\int_\Omega \xi(u(t,x))z(t,x)\,dt\,dx=-\int_0^T\int_\Omega \frac{1}{\sqrt{2\pi}} \int_\R \Fc(\xi)e^{i\omega u(t,x)} z(t,x)\,d\omega\,dt\,dx \nonumber\\
& \quad=\(\xi,-\frac{1}{\sqrt{2\pi}}\Fc^{-1}\(\frac{1}{(1+|\omega|^2)^s}\int_0^T\int_\Omega e^{-i\omega u(t,x)} z(t,x)\,dt\,dx\) \)_X =: \(\xi, S'(\Pi)^*v \)_X.\nonumber
\end{align}
Expressions \eqref{app-adjoint-eq} and \eqref{app-adjoint-int} complete the adjoint step  \eqref{seq-al:approx-adjoint} in the bi-level algorithms, explicitly given as
\begin{equation}\label{app-adjoint}
\begin{split}
&v  :=  u^j_{\kappa(j)}-\ydel, \\[0.5ex]
&S'(\Pitil)^*v  =  -\dfrac{1}{\sqrt{2\pi}}\Fc^{-1}\(\dfrac{1}{(1+|\omega|^2)^s}\int_0^T\int_\Omega e^{-i\omega u^j_{\kappa(j)}(t,x)} \ztil(t,x)\,dt\,dx\), \\[1ex]
&\begin{cases}
&-\dot{\ztil}-\nabla\cdot(a\nabla \ztil)+b\ztil+\Pitil'(u^j_{\kappa(j)})\ztil  = v \qquad  t\in(0,T)\\
&\ztil(T) = 0,
\end{cases}
\end{split}
\end{equation}
where $u^j_{\kappa(j)}=\sStil(\Pitil)$ is approximated by the lower-level. Before moving to the lower-level, we derive the adjoint for restricted data.

\subsubsection*{Upper-level adjoint -- final time measurement}
The adjoint state $z$ in this case is modified to
\begin{align*}
&\(G'(\Pi)\xi,v_T\)_Y 
= \int_\Omega p(T)v_T\,dx\\
&=\int_0^T\int_\Omega p\underbrace{\(-\dot{z} -\nabla\cdot(a\nabla z)+bz+\Pi'(u)z\)}_{=:0} \,dt\,dx + \int_\Omega p(T)v_T\,dx\\
&=\int_0^T\int_\Omega \(\dot{p} -\nabla\cdot(a\nabla p)+bu+\Pi'(u)p\)z\,dt\,dx+   \int_\Omega p(T)\underbrace{\(v_T-z(T)\)}_{=:0}\,dx=\langle-\xi(u),z\rangle
\end{align*}
for any $v_T\in H=L^2(\Omega)=Y$, with $p$ solving the linearized equation \eqref{app-linearized-eq}. The rest follows as in \eqref{app-adjoint-int}.
The upper-level adjoint with final time measurement  therefore takes the form
\begin{align}\label{app-adjoint-final}
&v:=u^j_{\kappa(j)}|_{t=T}-\ydel, \nonumber\\[0.5ex]
&S'(\Pitil)^*v = -\frac{1}{\sqrt{2\pi}}\Fc^{-1}\(\frac{1}{(1+|\omega|^2)^s}\int_0^T\int_\Omega e^{-i\omega u^j_{\kappa(j)}(t,x)} \ztil(t,x)\,dt\,dx\),\\[1ex]
&\begin{cases}
&-\dot{\ztil}-\nabla\cdot(a\nabla \ztil)+b\ztil+\Pitil'(u^j_{\kappa(j)})\ztil= 0 \qquad  t\in(0,T)\\
&\ztil(T)=v
\end{cases}
\end{align}
differing from  \eqref{app-adjoint} only in where the data $v$ enters the adjoint equation. Having completed the adjoint derivation for the upper-level, we proceed with the lower-level.

\subsubsection*{Lower-level adjoint}
The forward operator in the lower-level is the PDE model \eqref{app-eq} at any fixed $\Pi$
\begin{align*}
 F:\Vc\to\Uc^*\times H\qquad F(u)=F_\Pi(u):=
 \begin{pmatrix}
 \dot{u} -\nabla\cdot(a\nabla u)+bu+\Pi(u) \\
u|_{t=0}
 \end{pmatrix}
\end{align*}
Writing out the inner products in the state spaces $\Vc$ resp.~in the image space $\Uc^*\times H$ (c.f.~\eqref{app-setting}) and employing integration-by-parts in the Bochner spaces as well as the Riesz isomorphisms $D_U:U\to U^*$, $I_U:U^*\to U$, we write
\begin{align*}
&(u,z)_\Vc =\int_0^T (u,z)_U+(\dot{u},\dot{z})_{U^*}\,dt=\int_0^T\langle u, D_U z\rangle_{U,U^*}+\langle \dot{u},  I_U\dot{z}\rangle_{U^*,U}\,dt\\
&\qquad\quad=\int_0^T\langle u, D_U z- I_U\ddot{z}\rangle_{U,U^*}\,dt + (u(T),I_U\dot{z}(T))_H-(u(0),I_U\dot{z}(0))_H,\\[1ex]
&\(F'(u)\util, \begin{pmatrix}v\\h\end{pmatrix}\)_{\Uc^*\times H}=\bigg\langle F'(u)\util, \begin{pmatrix}I_Uv\\h\end{pmatrix}\bigg\rangle_{\Uc^*\times H,\,\Uc\times H}\\
&\qquad\quad=\int_0^T \Big\langle\util, \Big[ -d/dt- \nabla\cdot(a\nabla \cdot)+b+\Pi'(u)\Big]I_U v  \Big\rangle dt \\&\qquad\qquad + \(\util(T),I_U v(T)\)_H-\(\util(0),I_U v(0)-h\)_H=:\Big( \util, F'(u)^*(v;h) \Big)_\Vc
\end{align*}
for any $u$, $\util\in\Vc$. As $U=H_0^1(\Omega)$, we may choose $D_U=I_U^{-1}=-\Delta$, and so one obtains the bi-Laplacian wave equation
\begin{align}\label{app-adjoint1}
F'(u)^*(v;h)=:z\quad\text{with } \Bigg\{
\begin{array}{r@{\hspace{1ex}}c@{\hspace{1ex}}l}
-\ddot{z} + \Delta^2 z & = & \Delta \Big[ -d/dt- \nabla\cdot(a\nabla \cdot)+b+\Pi'(u)\Big]\Delta^{-1} v\\
\dot{z}(0) & =& v(0)+\Delta h; \quad \dot{z}(T)=v(T).
\end{array}
\end{align}
It is worth remarking that if the $\Uc$-inner product is considered, rather than the $\Vc$-inner product, then \eqref{app-adjoint1} boils down to a bi-Laplace elliptic equation, without the time factor. One also notices that if $a$, $b$ are constant, then the right-hand side of \eqref{app-adjoint1} simplifies to $-\dot{v}- a\Delta v +bv+\Delta(\Pi'(u)\Delta^{-1}v)$.

The adjoint in the lower-level has certain features in common with the all-at-once adjoint (c.f.~\cite{KaNg2022}), reflecting the fact that the bi-level method stands at the intersection between the reduced and the all-at-once approaches.

\subsubsection*{Algorithms for reaction identification}
Collecting all the ingredients, we build the algorithms for reaction law discovery.

\begin{algorithm}\caption{Reaction recovery from full data}\label{algorithm-app}
\{Upper-level\} Initialize $\Pitil^{\delta,0}\in B_R(\Pidag)$. Update:
\begin{equation}
\begin{split}\label{ex:update-pi}
& \Pitil^{\delta,j+1} = \Pitil^{\delta,j} -\dtilde{S'}(\Pitil^{\delta,j})^*\(\sStil(\Pitil^{\delta,j})-\ydel\) \qquad j\leq \jmax(\delta)
\end{split}
\end{equation}
\{Lower-level\} At each $\Pitil:=\Pitil^{\delta,j}$, sequentially initialize
\begin{equation*}
\boxed{\,\,\, u^j_0 := u^{j-1}_{\kappa(j-1)}\,\,\,}
\end{equation*}
\hphantom{\{Lower-level\}} then run:
\begin{equation}
\begin{split}
& u^j_{k+1}=u^j_k-F_u'(\Pitil,u^j_k)^*\(F(\Pitil,u^j_k)-\varphi\) \qquad k\leq \kappa(j) \\
& u^j_{\kappa(j)} = :\sStil(\Pitil)
\end{split}
\end{equation}
\hphantom{\{Lower-level\}} with PDE adjoint (c.f~\eqref{app-adjoint1})
\begin{align}\label{ex:adjoint-u}
&(v;h):= F(\Pitil,u^j_k)-\varphi\nonumber\\[1ex]
&F'(\Pitil,u)^*(v;h)=:z \nonumber\\[1ex]
&\begin{cases}
-\ddot{z} + \Delta^2 z= \Delta \Big[ -d/dt- \nabla\cdot(a\nabla \cdot)+b+\Pitil'(u)\Big]\Delta^{-1} v \\
\hspace{8cm} t\in(0,T)\\
\dot{z}(0)=v(0)+\Delta h; \qquad \dot{z}(T)=v(T)
\end{cases}
\end{align}
\{Upper-level\} Compute approximate adjoint (c.f.~\eqref{app-adjoint})
\begin{equation}\label{ex:adjoint-pi}
\begin{split}
&v:=u^j_{\kappa(j)}-\ydel\\[0.5ex]
&\dtilde{S'}(\Pitil)^*v = \frac{-1}{\sqrt{2\pi}}\Fc^{-1}\(\frac{1}{(1+|\omega|^2)^s}\int_0^T\int_\Omega e^{-i\omega u^j_{\kappa(j)}(t,x)} \ztil(t,x)\,dt\,dx\)\\[1ex]
&\begin{cases}
&-\dot{\ztil}-\nabla\cdot(a\nabla \ztil)+b\ztil+\Pitil'(u^j_{\kappa(j)})\ztil= v \qquad  t\in(0,T)\\
&\ztil(T)=0
\end{cases}
\end{split}
\end{equation}

\end{algorithm}
\begin{algorithm}[H]\caption{Reaction recovery from terminal time data}\label{algorithm-app-final}\,\\
Identical to Algorithm \ref{algorithm-app}, except:\\[1ex]
\{Upper-level\} Compute approximate adjoint state (c.f.~\eqref{app-adjoint-final})
\begin{equation*}
\begin{split}
&v:={u^j_{\kappa(j)}}|_{t=T}-\ydel\\[0.5ex]
&\dtilde{S'}(\Pitil)^*v = \frac{-1}{\sqrt{2\pi}}\Fc^{-1}\(\frac{1}{(1+|\omega|^2)^s}\int_0^T\int_\Omega e^{-i\omega u^j_{\kappa(j)}(t,x)} \ztil(t,x)\,dt\,dx\)\\[1ex]
&\begin{cases}
&-\dot{\ztil}-\nabla\cdot(a\nabla \ztil)+b\ztil+\Pitil'(u^j_{\kappa(j)})\ztil= 0 \qquad  t\in(0,T)\\
&\ztil(T)=v
\end{cases}
\end{split}
\end{equation*}

\end{algorithm}


\subsection{Tangential cone condition and discussion}\label{sec:app-tcc}

Having established the details of the reconstruction procedure, we discuss some important assumptions before carrying out our numerical experiments. The most challenging aspect of Assumption \ref{summary-ass-bi} to establish is the tangential cone condition \eqref{summary-ass-up-1}, which quantifies nonlinearity of the forward map. The tangential cone condition was first introduced in the seminal work \cite{scherzer95}, and recently developed into several extended versions by \cite{Kindermann17}. This structural condition is the key to ensure convergence of iterative gradient-based methods, such as Landweber, Newton type methods, their (projected, constrained, accelerated, frozen) variations as well as of the bi-level schemes proposed here. 

The tangential cone condition was historically developed for ill-posed inverse problems with compact forward operators, but also has certain relation to other nonlinearity conditions popular in optimization, such as convexity and the Polyak-Lojasiewicz condition \cite{nguyen24}. 
We refer to \cite{TCC21} for a general verification strategy of such condition for parameter identification in  nonlinear parabolic PDEs, to \cite{HoffmanWaldNguyen:2021} for linear elliptic inverse problems, to \cite{Nakamura:MRE21, HubmerScherzer:TCC18} for elastography, to \cite{Rieder:TCC21, EllerRolandRieder24} for full wave form inversion, to \cite{Kindermann21} for electrical impedance tomography, and, recently, to \cite{ScherzerHofmannNashed, holler22learning_parameter_id} for deep neural networks. Extending this line of work, we now investigate the tangential cone condition for reaction problems.

\subsubsection*{Verification of the tangential cone condition}
Denote by $\Pidag$ the exact unknown reaction law. The tangential cone condition \eqref{summary-ass-up-1} for the parameter-to-observation map $G=L\circ S$ then reads as
\begin{equation}\label{app-tcc1}
\begin{split}
\|G(\Pi)-G(\Pidag)-G'(\Pi)(\Pi-\Pidag)\|_\Yc\leq C_{tc}\|G(\Pi)-G(\Pidag)\|_\Yc\qquad &C_{tc}<1\\ &\forall\Pi\in  B_R(\Pidag)
\end{split}
\end{equation}
with $C_{tc}:=1-(M_R^2+\mu_R)/2\in(0,1)$. Clearly, this condition constrains the linearized error of $G$ around the ground truth $\Pidag$. For full observation $L=\text{Id}$, the relation \eqref{app-tcc1} becomes
\begin{equation}\label{app-tcc2}
\begin{split}
\|S(\Pi)-S(\Pidag)-S'(\Pi)(\Pi-\Pidag)\|_\Yc\leq C_{tc}\|S(\Pi)-S(\Pidag)\|_\Yc\qquad &C_{tc}<1\\ &\forall\Pi\in  B_R(\Pidag).
\end{split}
\end{equation}
In the following, we will establish the tangential cone condition for full data; the same question for final time data is a subject for future research.

\emph{Step 1.} We begin by constructing the equation associated with the left hand side of \eqref{app-tcc2} using \eqref{app-eq} and its linearization \eqref{app-linearized-eq}, obtaining
\begin{equation*}
\begin{array}{ccc}
\begin{array}{r@{\hspace{1ex}}c@{\hspace{1ex}}l}
S(\Pi)& =:& u \\  
S(\Pidag) & =: & v \\
S'(\Pi)(\Pi-\Pidag) & =:& w  
\end{array}
\begin{array}{ll}
& \dot{u} -\nabla\cdot(a\nabla u)+bu+\Pi(u) = \phi \\
& \dot{v} -\nabla\cdot(a\nabla v)+bv+\Pidag(v) = \phi \\
& \dot{w} -\nabla\cdot(a\nabla w)+bw+\Pi'(u)w = -(\Pi-\Pidag)(u)
\end{array}
\begin{array}{r@{\hspace{1ex}}c@{\hspace{1ex}}l}
u|_{t=0} & = & u_0\\
v|_{t=0} & = & u_0\\
w|_{t=0} & = & 0.
\end{array}
\end{array}
\end{equation*}
Hence,
\begin{align*} 
& W:=u-v-w \qquad\text{solves}\\
&\begin{cases}\dot{W} -\nabla\cdot(a\nabla W)+bW+\Pi'(u)W=-\Pi(u)+\Pidag(v)+\Pi'(u)u-\Pi'(u)v+ (\Pi-\Pidag)(u)
\\W|_{t=0}=0.
\end{cases} \nonumber
\end{align*}

\emph{Step 2.} Next, \eqref{summary-ass-bi-2} supposes that for the adjoint linearized equation, the solution in $\Uc=L^2(0,T;H^1_0(\Omega))$ depends continuously on the source term in $\Uc^*$. 
Employing continuity of the various relevant embeddings, we have
\begin{align*}
&1/(C_{\Uc\to\Yc}C_{f_u})\|W\|_\Yc\\
&\leq1/C_{f_u}\|W\|_{\Uc}\leq
\|\Pidag(v)-\Pidag(u)+\Pi'(u)(u-v)\|_{L^2(H^{-1})}\\
&\leq  \big\| \int_0^1 \Pi'_\dagger\big(u+\lambda(v-u)\big)\,d\lambda(v-u) - \Pidag'(u)(v-u)\big\| + \big\|\Pi'_\dagger(u)(v-u)-\Pi'(u)(u-v)\big\|\\
&= \big\| \int_0^1\int_0^1 \Pi''_\dagger\big(u+\eta\lambda(v-u)\big)\,d\lambda\,d\eta(v-u)^2\big\|+ \big\|(\Pi'_\dagger-\Pi')(u)(u-v)\big\|\\
&\leq C_{H^1\to L^1}\|\Pi''_\dagger\|_{L^\infty(\R)}\|(v-u)^2\|_{L^2(L^1)} + C_{H^1\to L^2}\|\Pi'_\dagger-\Pi'\|_{L^\infty(\R)}\|v-u\|_{L^2(L^2)}\\
&\leq C_{\Vc\to C(H)}C_{H^1\to L^1}\|\Pi_\dagger\|_{W^{2,\infty}(\R)}\|v-u\|_\Vc\|v-u\|_\Yc \\&\qquad\qquad+ C_{X\to W^{1,\infty}}C_{H^1\to L^2}\|\Pi_\dagger-\Pi\|_X\|v-u\|_\Yc
\end{align*}
From this, we see see that the ground truth $\Pidag$ possesses higher regularity than of $X$.

\emph{Step 3.} Further estimating $\|v-u\|_\Vc$ using bounded derivative assumption \eqref{summary-ass-bi-1} yields
\begin{align*}
&\|v-u\|_\Vc=\|S(\Pidag)-S(\Pi)\|_\Vc\leq \sup_{\Pi\in B_R(\Pidag)}\|S'(\Pi)\|_{X\to\Vc} \|\Pi_\dagger-\Pi\|_X\leq \overline{M}_S R.
\end{align*}
By this, we are now ready to show \eqref{app-tcc2}:
\begin{align*}
&\|S(\Pi)-S(\Pidag)-S'(\Pi)(\Pi-\Pidag)\|_\Yc=\|W\|_\Yc \nonumber\\
&\leq C_{\Uc\to\Yc}C_{f_u}\big( C_{\Vc\to C(H)}C_{H^1\to L^1}\|\Pi_\dagger\|_{W^{2,\infty}(\R)}\overline{M}_S +  C_{X\to W^{1,\infty}}C_{H^1\to L^2}\big)R \nonumber\\&\qquad\cdot\|S(\Pi)-S(\Pidag)\|_\Yc \nonumber\\[0.5ex]
&=:C_{tc}\|S(\Pi)-S(\Pidag)\|_\Yc
\end{align*}
with $C_{tc}<1$ for sufficiently small radius $R>0$ of the ball $B_R(\Pidag)$. The tangential cone condition has thus been established.

We conclude this section by some remarks on uniqueness. The tangential cone condition \eqref{app-tcc2} induces uniqueness of the minimum norm solution $\Pi_\dagger$. Together with the null space condition  
$ \mathcal{N}(G'(\Pidag))\subset\mathcal{N}(G'(\Pi))$ for all $\Pi\in B_R(\Pidag)$, the reconstruction sequence converges to this minimum norm solution by \cite{KalNeuSch08}. On the subject of unique identifiability with restricted data, e.g.~boundary, time trace or final time measurement, we refer to the work in \cite{DuChateauRundell:1985,PilantRundell:1986}.

\subsection{Numerical results}\label{sec:numerics}
This section implements the bi-level algorithms for several reaction-diffusion-based applications. We select various models -- enzyme kinetics, combustion, biology -- to examine:
\begin{itemize}[label=]
\item Fisher equation: $\Pi(u)=4u(1 - u)$ 
\item Lane-Emden equation: $\Pi(u)=2u/(1+u+4u^2)$ 
\item Zeldovic-Frank-Kamenetskii (ZFK) equation: $\Pi(u)=4u(1-u)e^{-2(1-u)}$ 
\end{itemize}
with diffusivity part  $\dot{u}-\Delta u$.

For implementation setup, we partition the space interval $[0,1]$ into $101$ regular grids points, and the time line $[0,0.1]$ into $51$ moments. The Laplacian $\Delta$ in  \eqref{ex:adjoint-u} and the bi-Laplacian $\Delta^2$ in \eqref{ex:adjoint-pi} are approximated by central difference quotients. To invert these differential operators, we invoke the Python's inbuilt LU decomposition. 

Regarding time integration in the adjoint \eqref{ex:adjoint-pi}, a trapezoid rule is employed. It is worthwhile to remark that the reaction $\Pi$ in these examples are functions of the state $u$, itself ranging within the interval $[-1,1]$; we therefore discretize $\text{ran}(u)$ into $50$ equidistant points. Python's inbuilt inverse fast Fourier transform (iFFT) is then called to complete the last step of the adjoint. 

The derivative of $\Pi$ is needed for the adjoint equations. This derivative is initially approximated by finite difference at the collocation points in the range of $u$. Then, the linear interpolation $\Pi'(\text{ran}(u))\mapsto\Pi'(u(t,x))$ re-grids it to the physical domain $(t,x)$ before inputting it into the adjoint PDEs. A more precise -- albeit more costly -- option is to perform several inverse normal Fourier transforms. 
Finally, inversion is performed with full measurement (Algorithm \ref{algorithm-app}) and with final time measurement (Algorithm \ref{algorithm-app-final}) with errors plot in $L^2$-norms.\\

\begin{figure}[htb!]
\centering
\includegraphics[trim=3.5cm 0cm 0cm 0cm,clip=true,scale=0.4]{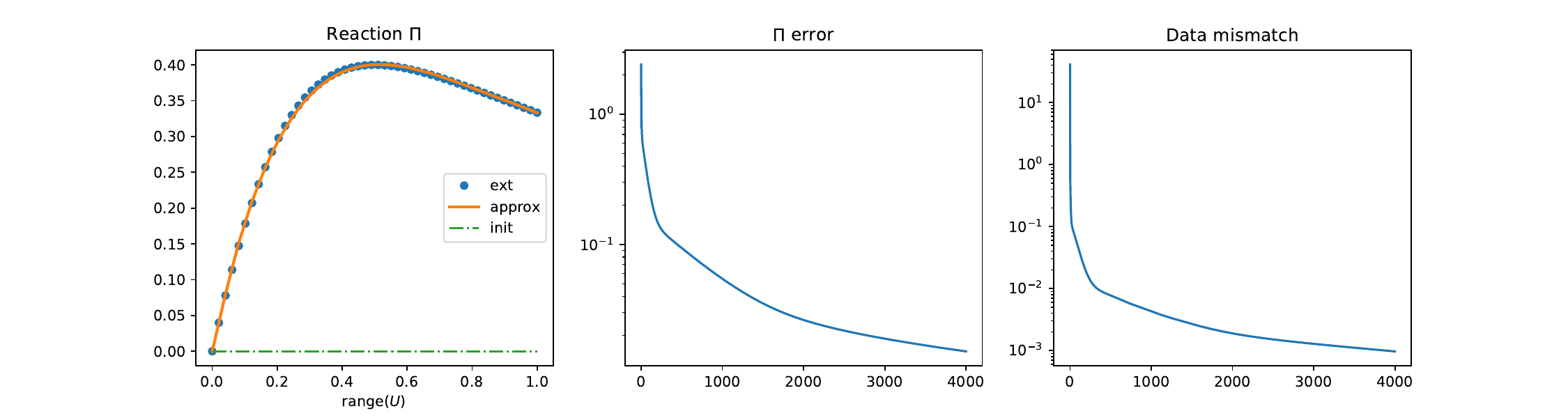}\\
\includegraphics[trim=3.5cm 0cm 2cm 0cm,clip=true,scale=0.4]{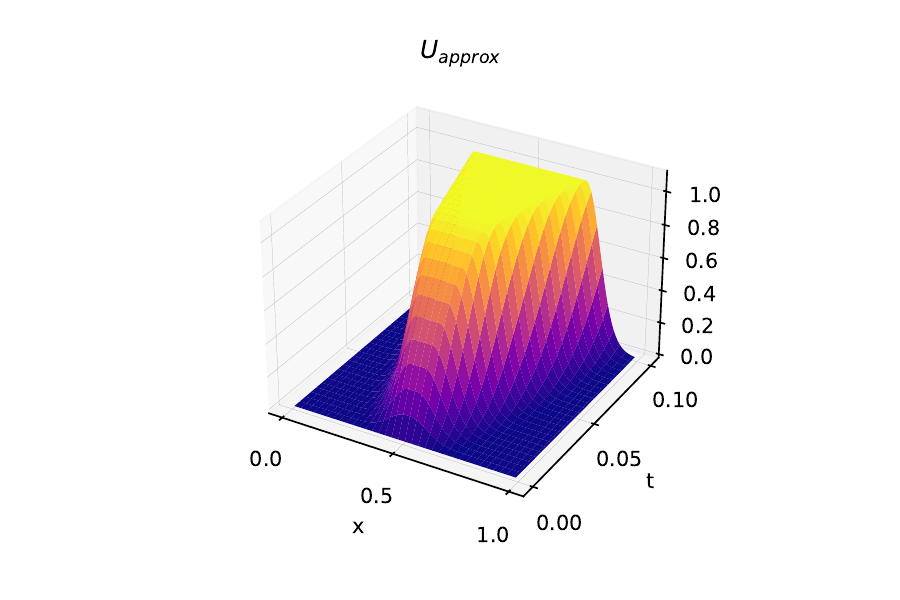}
\includegraphics[trim=3.5cm 0cm 2cm 0cm,clip=true,scale=0.4]{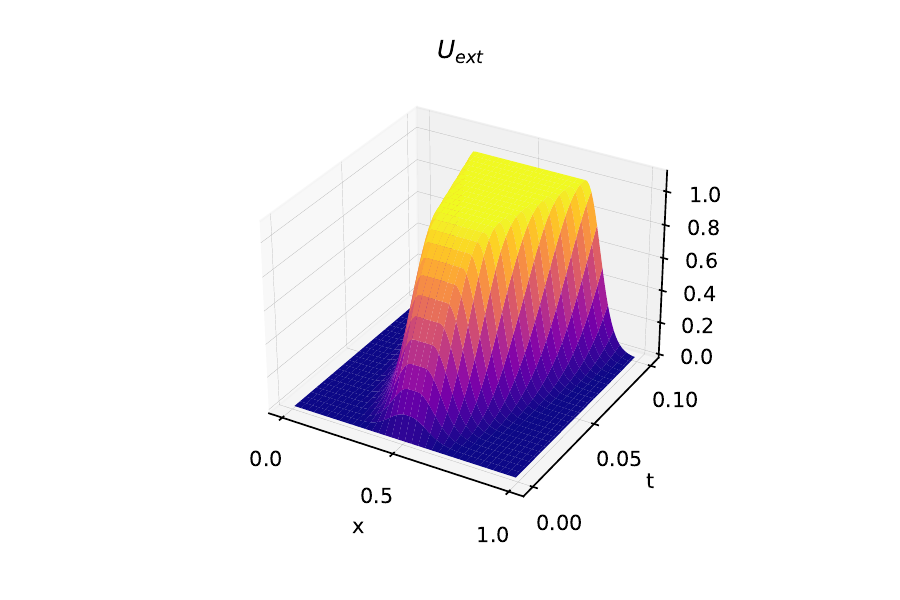}\\
\includegraphics[trim=3.5cm 0cm 2cm 0cm,clip=true,scale=0.4]{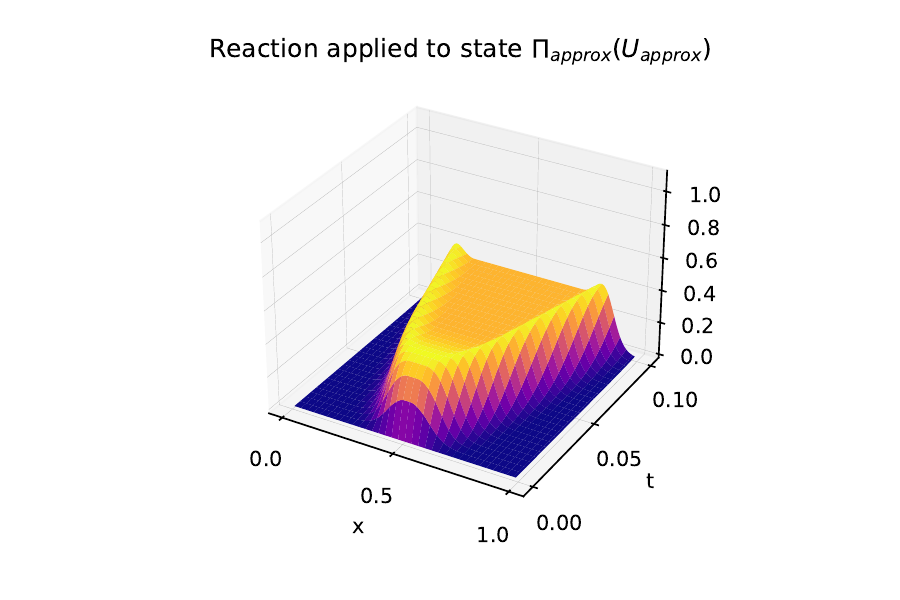}
\includegraphics[trim=3.5cm 0cm 2cm 0cm,clip=true,scale=0.4]{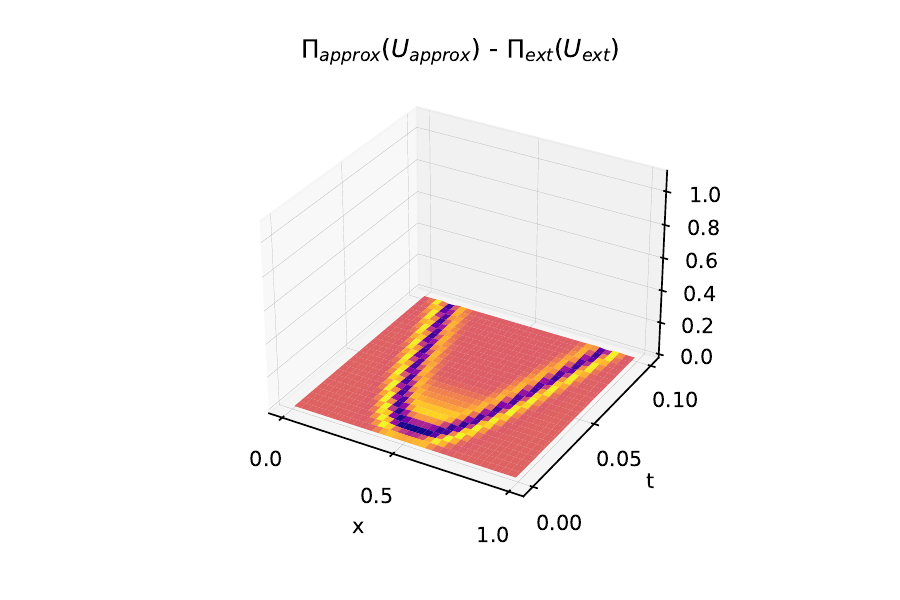}
\caption{Lane-Emden -- Full data. Top: reaction reconstruction. Middle: state reconstruction. Bottom: reaction applied to state.}\label{fig-full}
\end{figure}

\noindent\textbf{Full data. }Fig.~\ref{fig-full} presents the bi-level reconstruction result for the Lane-Emden equation. The top left panel, which we will refer to as panel $(1,1)$, displays the reconstructed reaction $\Pi$ plotted over the range of $u$ as the $x$-axis. Despite the very poor initial guess $\Pi_\text{init}=0$, the estimation shows an excellent fit to the ground truth $\Pi^\dagger$ after $4000$ iterations. As a consequence, the error in panel $(1,2)$ decays. 

Regarding the state, the estimation in panel $(2,1)$ well approximates the exact one, panel $(2,2)$. We emphasize that the lower-level is initiated at zero, without any prior information on $u^\dagger$, whereupon -- as Algorithm \ref{algorithm3} suggests -- we sequentially employ the output of each lower-trajectory as the stating point to compute the next one. The application of the reaction law to the state $\Pi(u)$ modifies $u$ in a nonlinear manner (panel $(3,1)$). 
Panel $(3,2)$ reports the error, which is close to being a zero plane. Overall convergence is confirmed by the descent of the data mismatch in panel $(1,3)$.\\

\begin{figure}[htb!]
\centering
\includegraphics[trim=3.5cm 0cm 0cm 0cm,clip=true,scale=0.4]{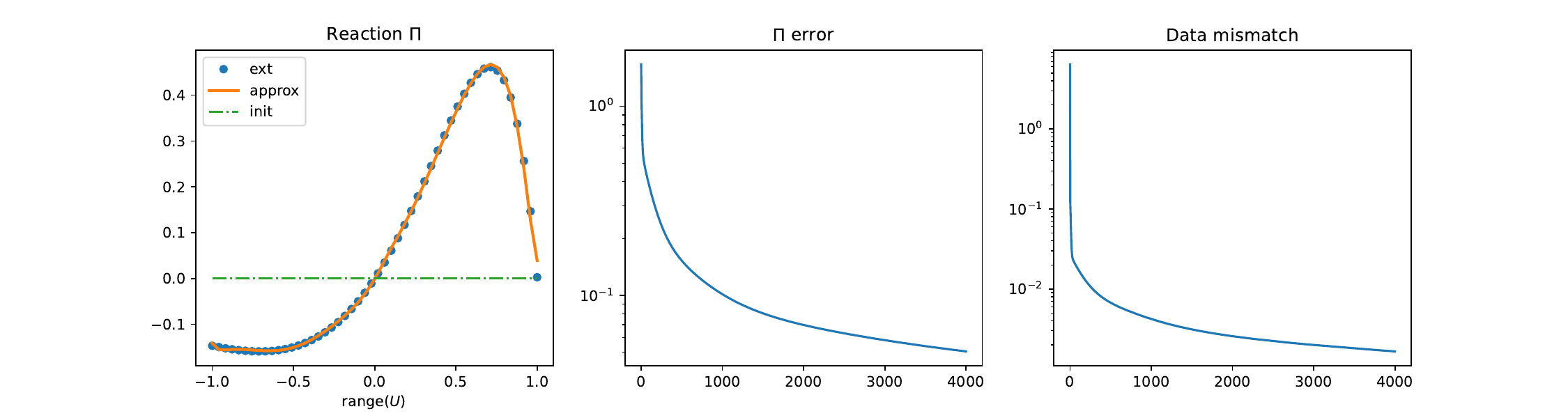}\\
\includegraphics[trim=3.5cm 0cm 2cm 0cm,clip=true,scale=0.4]{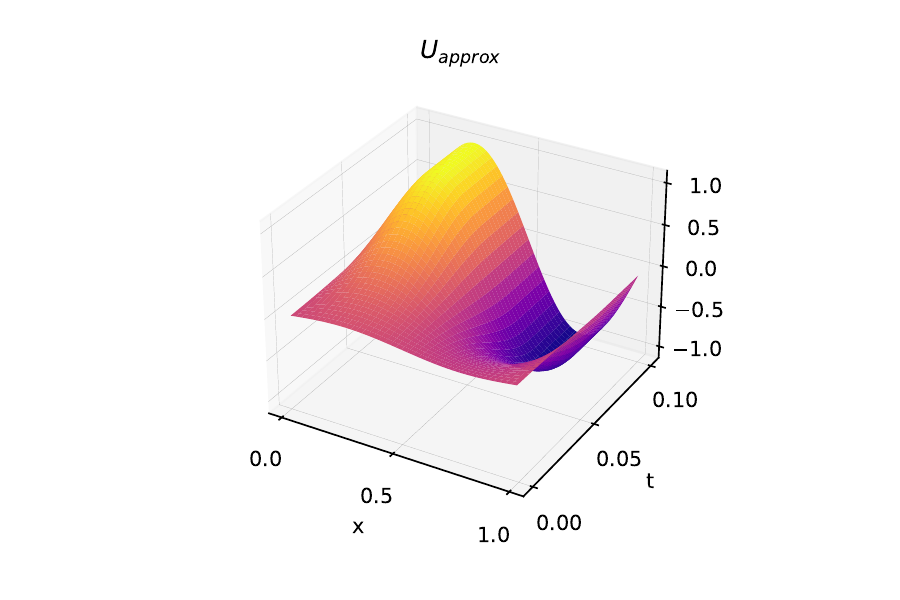}
\includegraphics[trim=3.5cm 0cm 2cm 0cm,clip=true,scale=0.4]{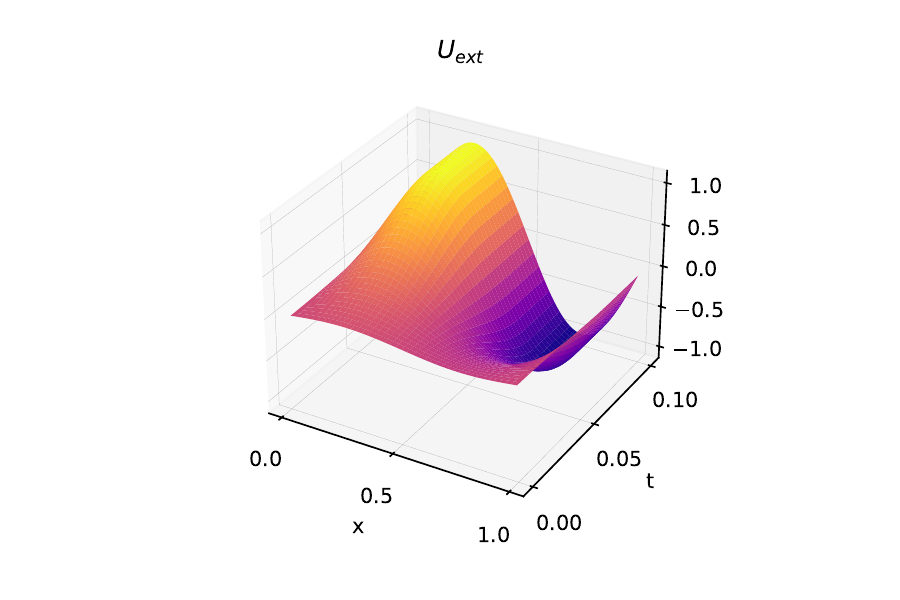}
\caption{ZFK -- Terminal time data. Top: reaction reconstruction. Middle: state reconstruction. }\label{fig-final}
\end{figure}

\noindent\textbf{Final time data. }
Algorithm \ref{algorithm-app-final} with final time measurement $y=u(T)$ can successfully recover the reaction and state, as evidenced by Fig.~\ref{fig-final}. In panel $(1,1)$, we observe that the reconstructed $\Pi$ nicely captures the features of the true nonlinearity in the ZFK equation. In addition, the approximate state also matches the ground truth, as seen in panels $(2,1), (2,2)$. The decay in panels $(1,2)$ and $(1,3)$ numerically confirms convergence of the bi-level method. \\

\noindent\textbf{Standard vs.~sequential bi-level. }We now compare Algorithm \ref{algorithm2} and  Algorithm \ref{algorithm3} for ZFK model with final time measurement. Fig.~\ref{fig-biseq} top and bottom rows respectively displays the internal operation of these schemes. The left panels plots the number of lower-iterations ($y$-axis) over each upper-iteration ($x$-axis). With the same precision, the standard scheme ran with the same amount of inner-loops. Remarkably, the sequential version requires much fewer iterations; in fact, the number of required lower-iterations decays rapidly, with no more than a single lower-iteration needed after about half the total number of upper-iterations. In addition, the PDE residual panels on the right clearly depicts acceleration of the sequential version. 

Regarding the reconstruction, we observe that in Fig.~\ref{fig-biseq-2} the sequential version outcome behaves much like the standard one, demonstrating the fact that in spite of the significantly lower computational cost of the sequential algorithm, the inversion quality is similar to that of the non-sequential bi-level algorithm.\\

\begin{figure}[htb!]
\centering
\includegraphics[trim=1cm 0cm 1.5cm 0.5cm,clip=true,scale=0.36]{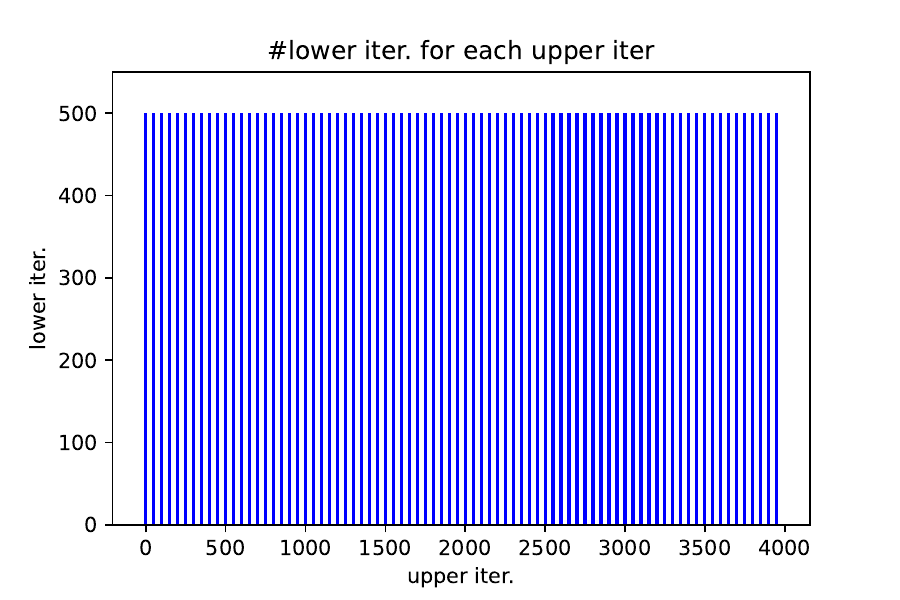}
\includegraphics[trim=1cm 0cm 1.5cm 0.5cm,clip=true,scale=0.36]{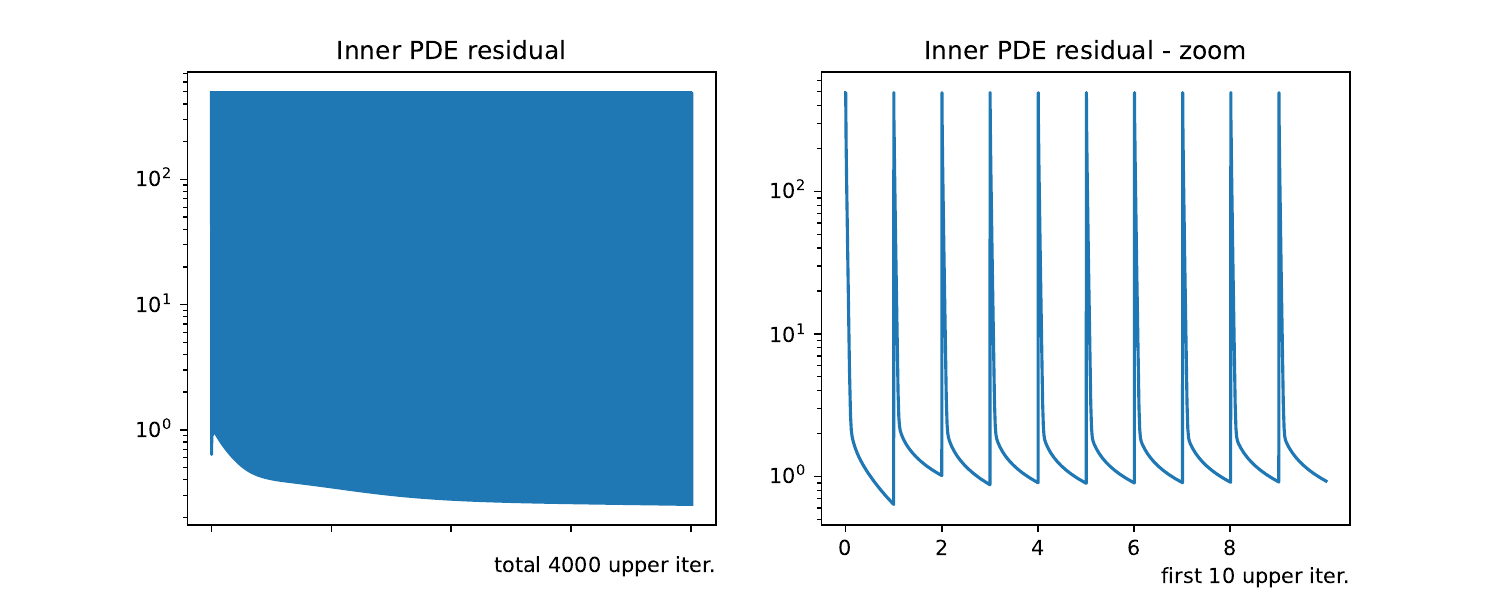}
\vspace{0.5cm}
\includegraphics[trim=1cm 0cm 1.5cm 0.5cm,clip=true,scale=0.36]{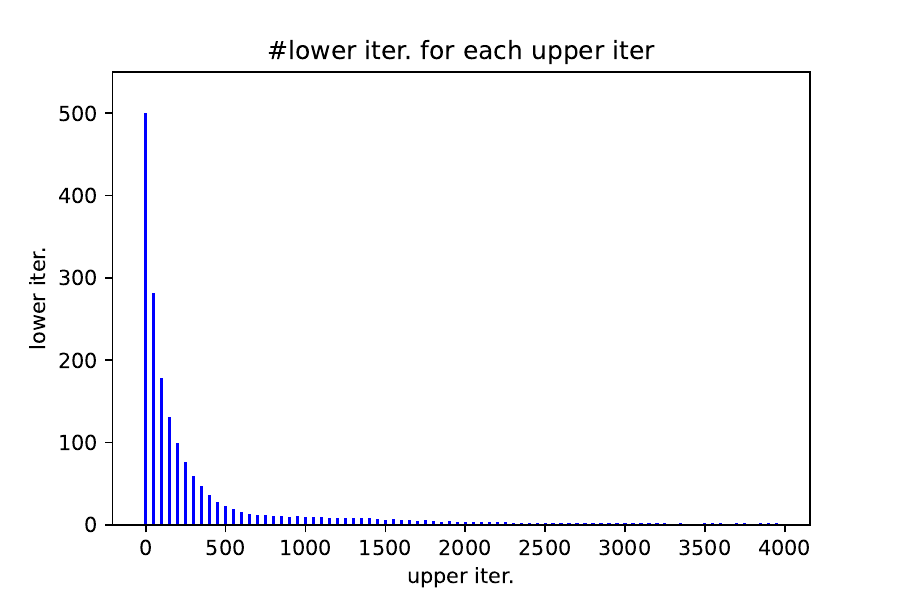}
\includegraphics[trim=1cm 0cm 1.5cm 0.5cm,clip=true,scale=0.36]{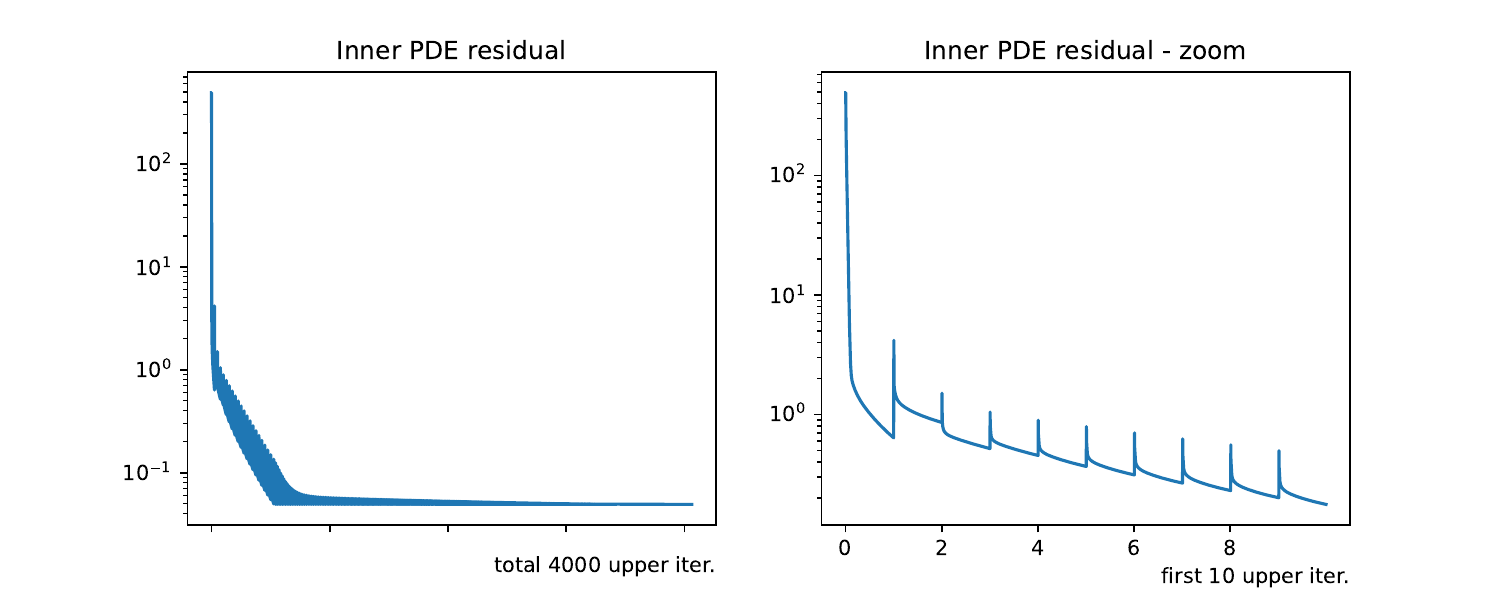}
\caption{Standard (top) vs.~sequential (bottom) bi-level operation. }\label{fig-biseq}
\end{figure}

\begin{figure}[H]
\centering
\includegraphics[trim=3.5cm 0cm 14cm 0.5cm,clip=true,scale=0.4]{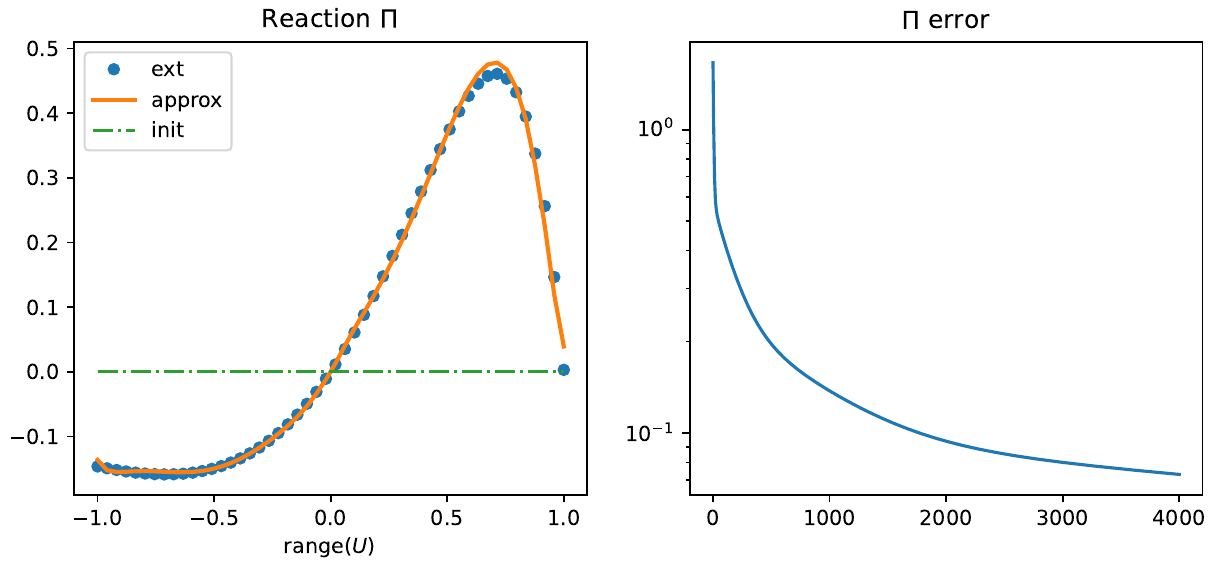}
\includegraphics[trim=3.5cm 0cm 24cm 0.5cm,clip=true,scale=0.4]{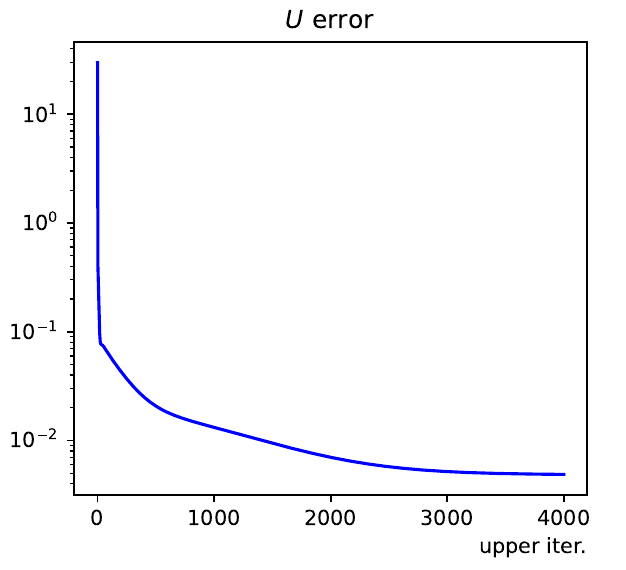}
\vspace{0.5cm}
\includegraphics[trim=3.5cm 0cm 14cm 0.5cm,clip=true,scale=0.4]{figures/fig-file-ex2-final-f}
\includegraphics[trim=3.5cm 0cm 24cm 0.5cm,clip=true,scale=0.4]{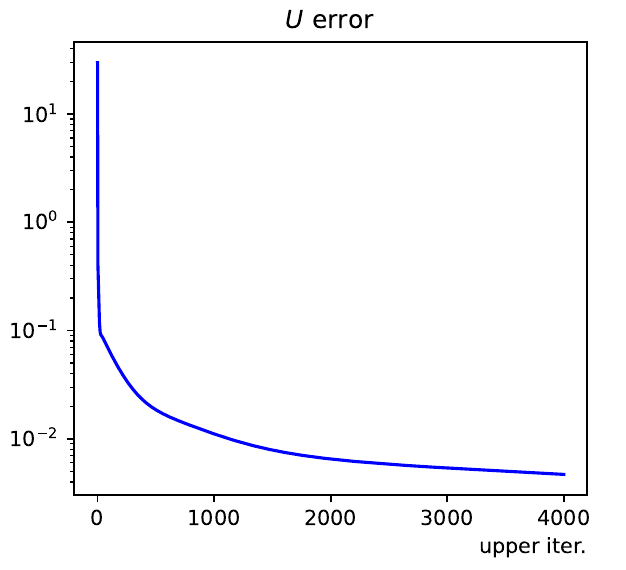}
\caption{Standard (top) vs sequential (bottom) bi-level reconstruction.}\label{fig-biseq-2}
\end{figure}

\begin{figure}[htb!]
\centering
\includegraphics[trim=3.5cm 0cm 3cm 0.5cm,clip=true,scale=0.35]{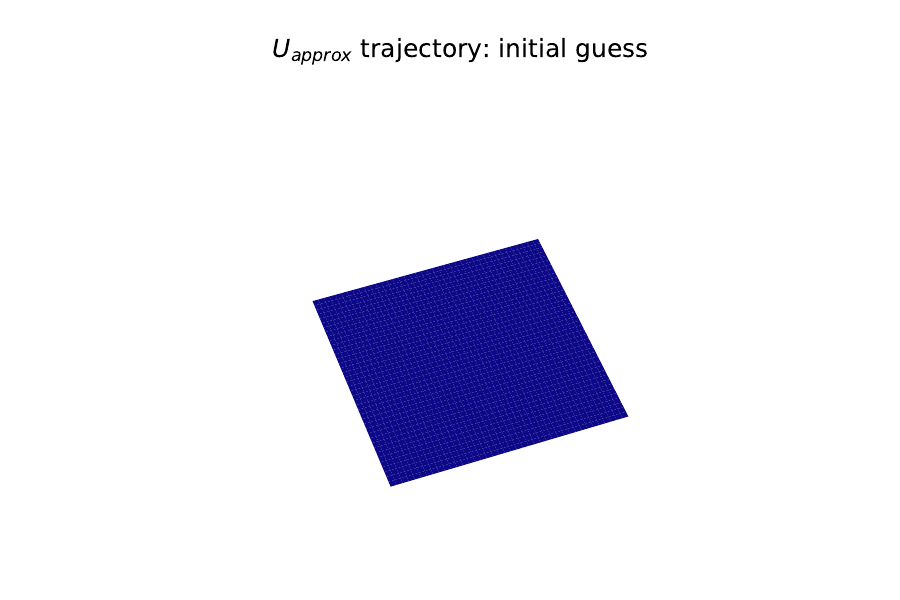}
\includegraphics[trim=3.5cm 0cm 3cm 0.5cm,clip=true,scale=0.35]{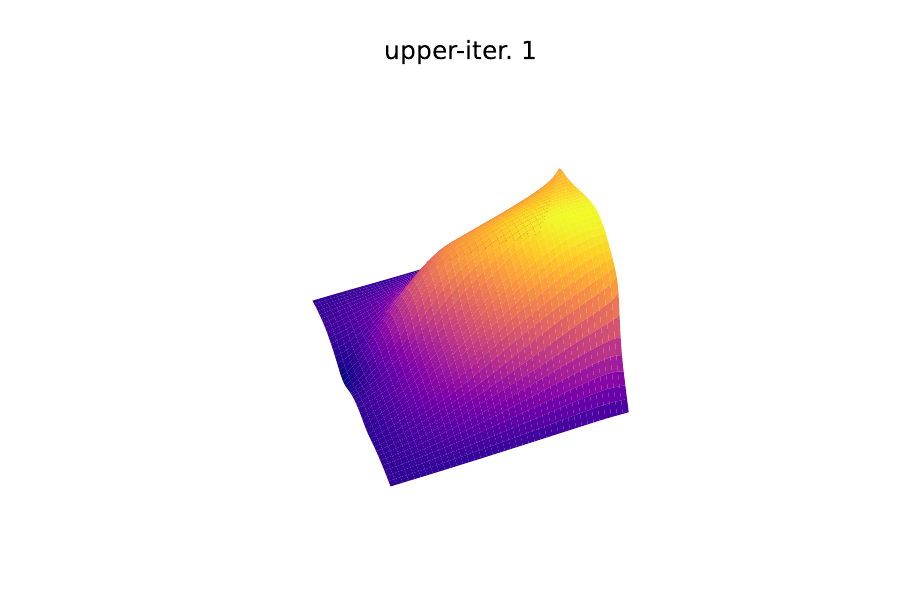}
\includegraphics[trim=3.5cm 0cm 3cm 0.5cm,clip=true,scale=0.35]{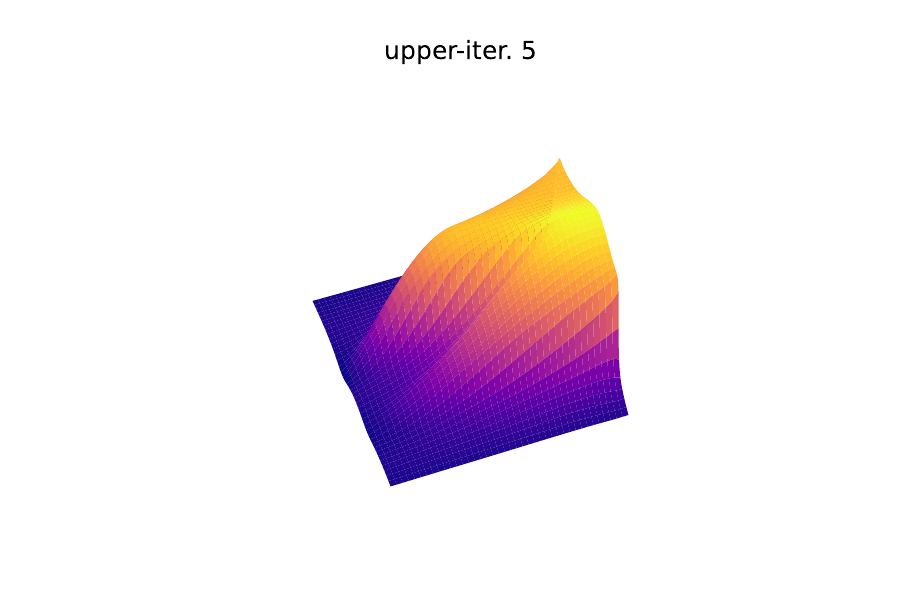}
\includegraphics[trim=3.5cm 0cm 3cm 0.5cm,clip=true,scale=0.35]{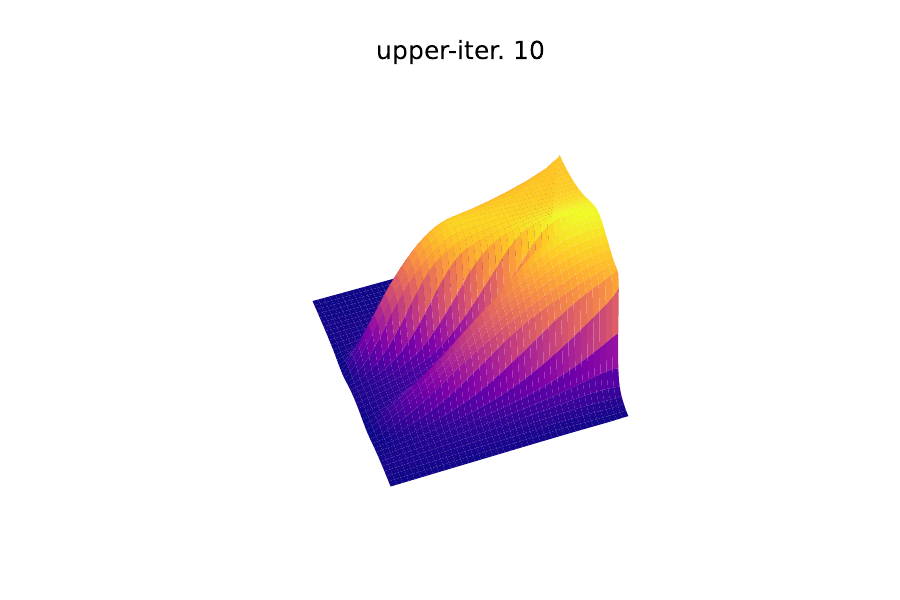}

\includegraphics[trim=3.5cm 0cm 3cm 0.5cm,clip=true,scale=0.35]{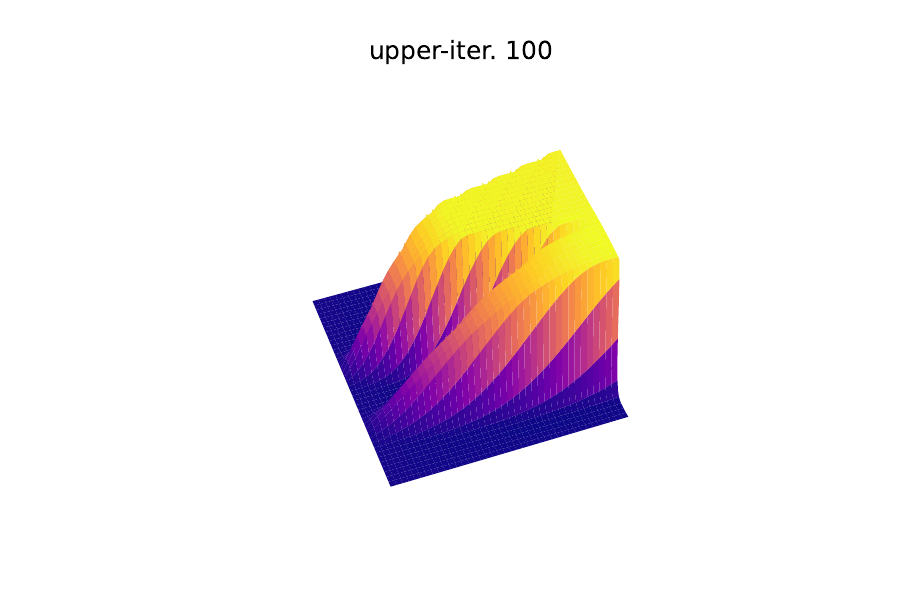}
\includegraphics[trim=3.5cm 0cm 3cm 0.5cm,clip=true,scale=0.35]{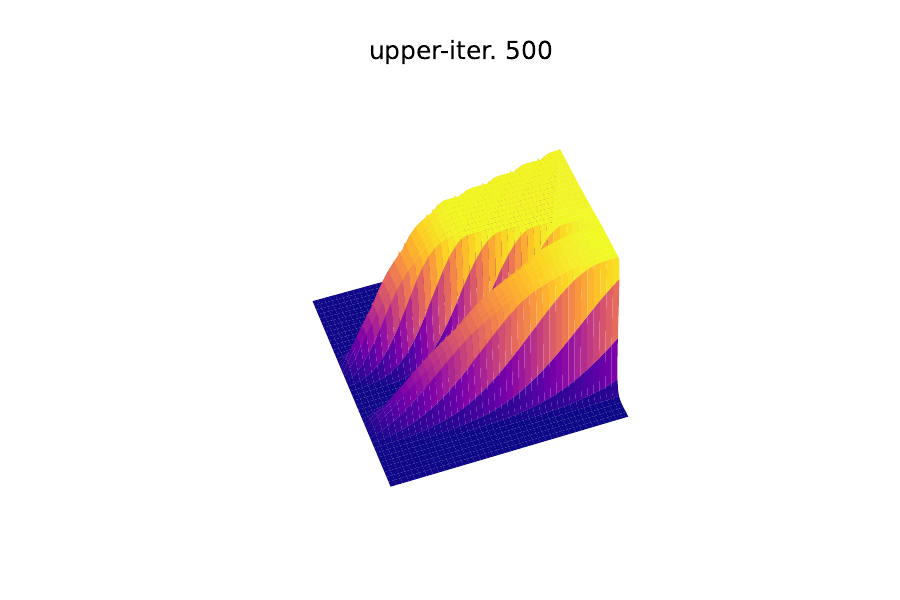}
\includegraphics[trim=3.5cm 0cm 3cm 0.5cm,clip=true,scale=0.35]{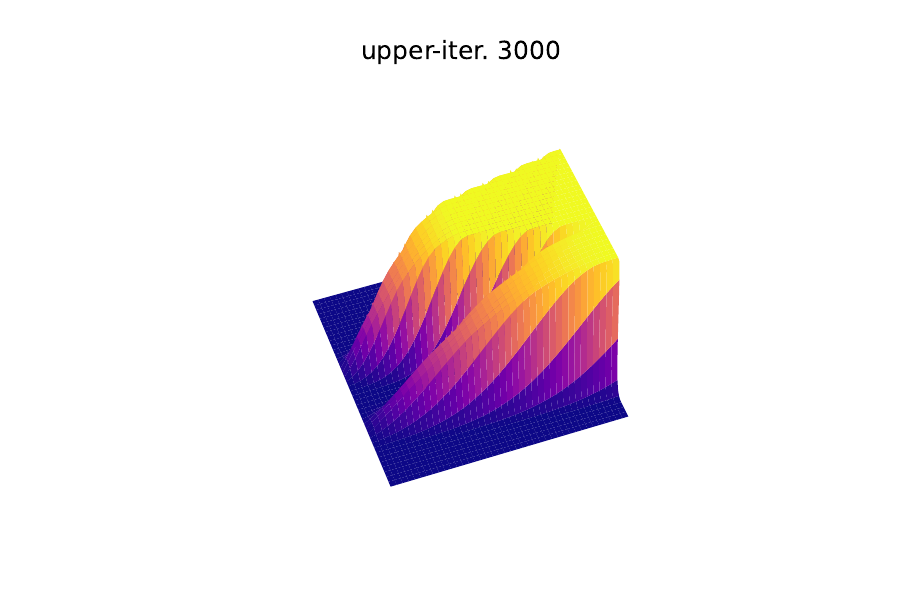}
\includegraphics[trim=3.5cm 0cm 3cm 0.5cm,clip=true,scale=0.35]{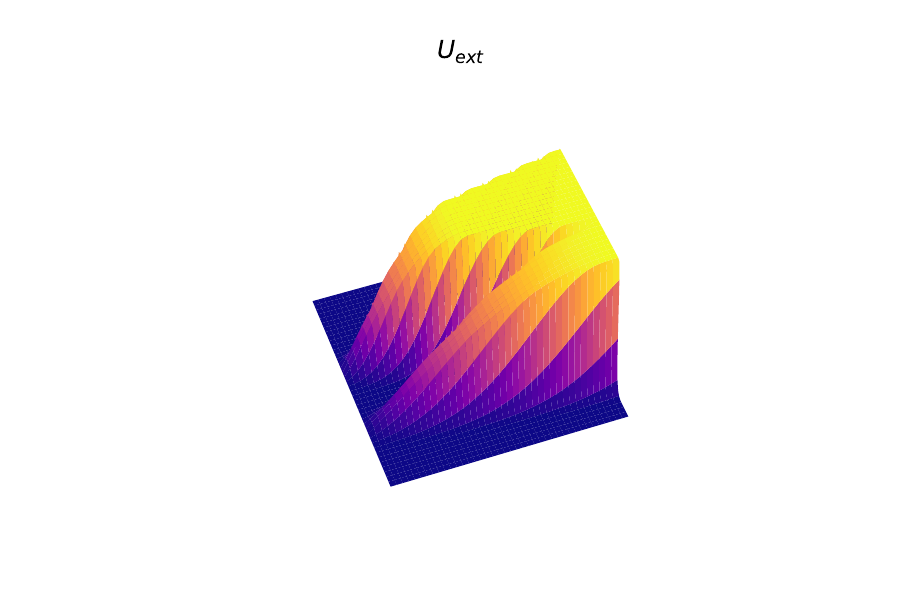}

\caption{Fisher -- State trajectory. Each lower-level ran 10 iterations.}\label{fig-traject-U}
\end{figure}

\begin{figure}[H]
\centering
\includegraphics[trim=2.8cm 0cm 1.7cm 0cm,clip=true,scale=0.35]{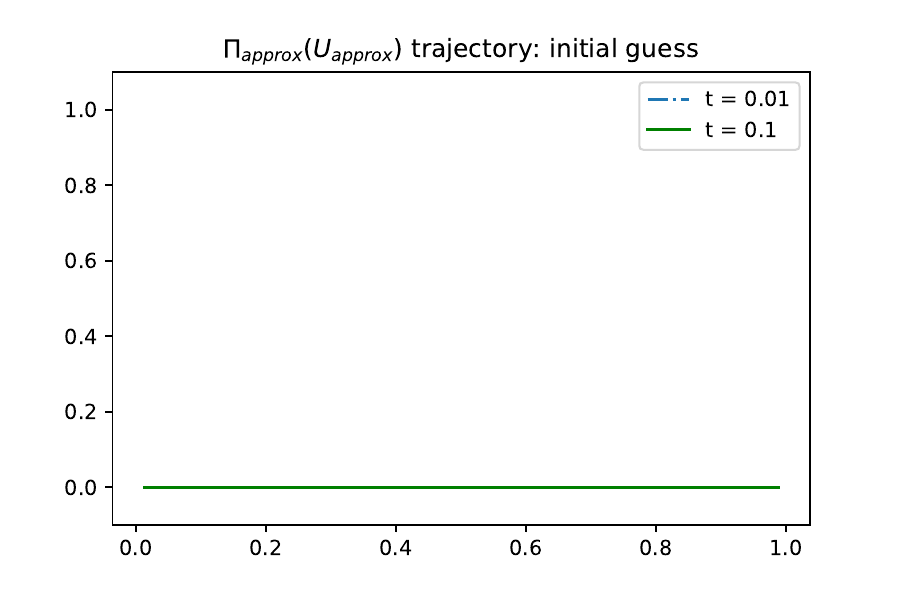}
\includegraphics[trim=2.8cm 0cm 2.5cm 0cm,clip=true,scale=0.35]{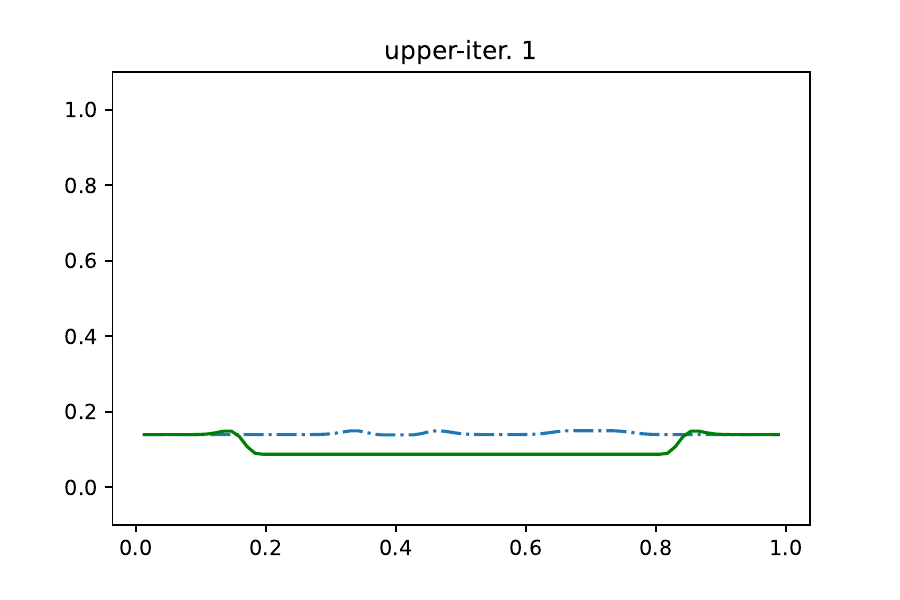}
\includegraphics[trim=2.8cm 0cm 2.5cm 0cm,clip=true,scale=0.35]{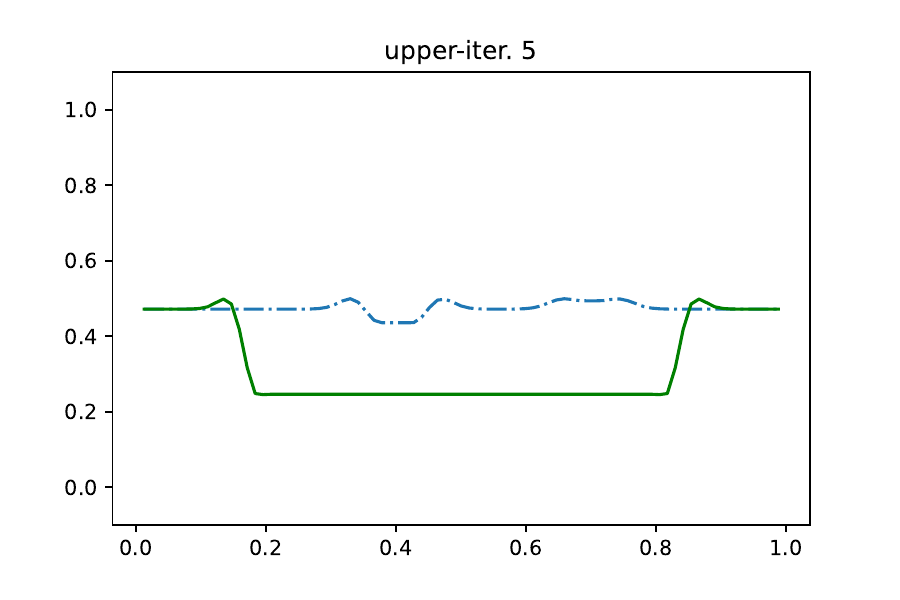}

\includegraphics[trim=2.8cm 0cm 2.5cm 0cm,clip=true,scale=0.35]{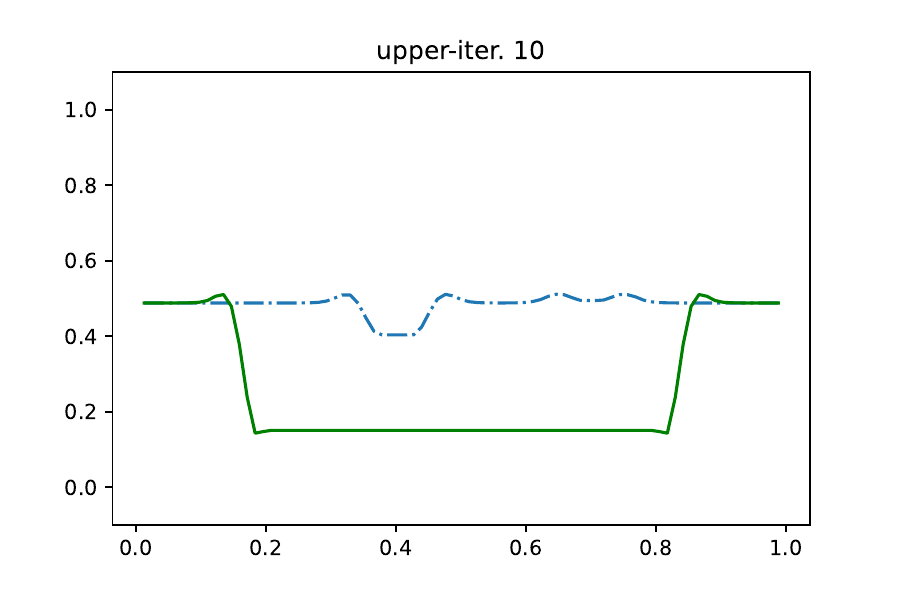}
\includegraphics[trim=2.8cm 0cm 2.5cm 0cm,clip=true,scale=0.35]{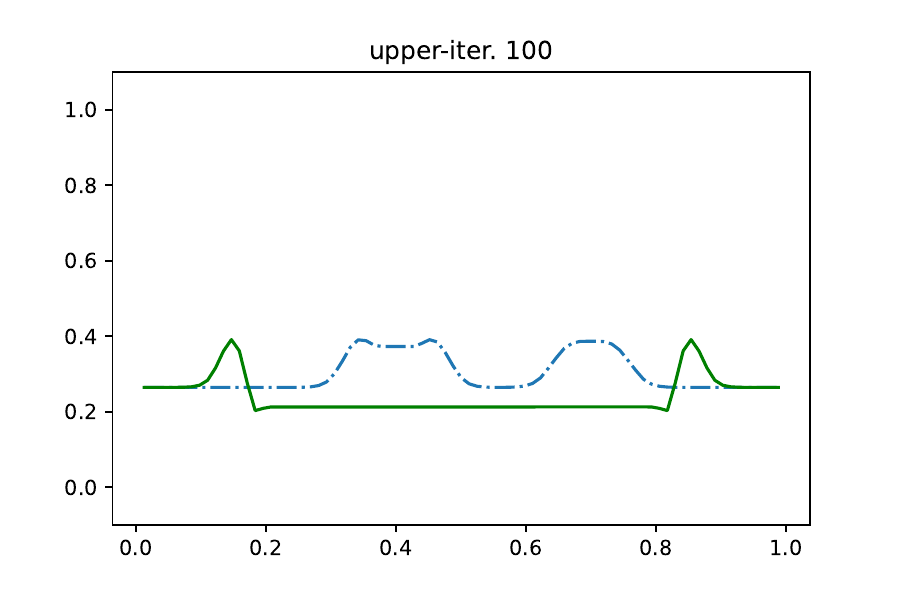}
\includegraphics[trim=2.8cm 0cm 2.5cm 0cm,clip=true,scale=0.35]{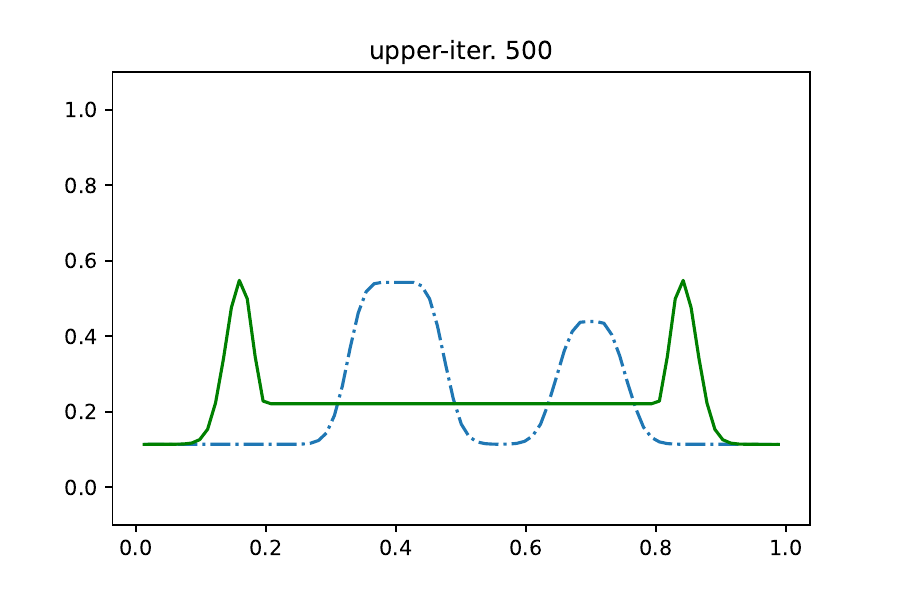}

\includegraphics[trim=2.8cm 0cm 2.5cm 0cm,clip=true,scale=0.35]{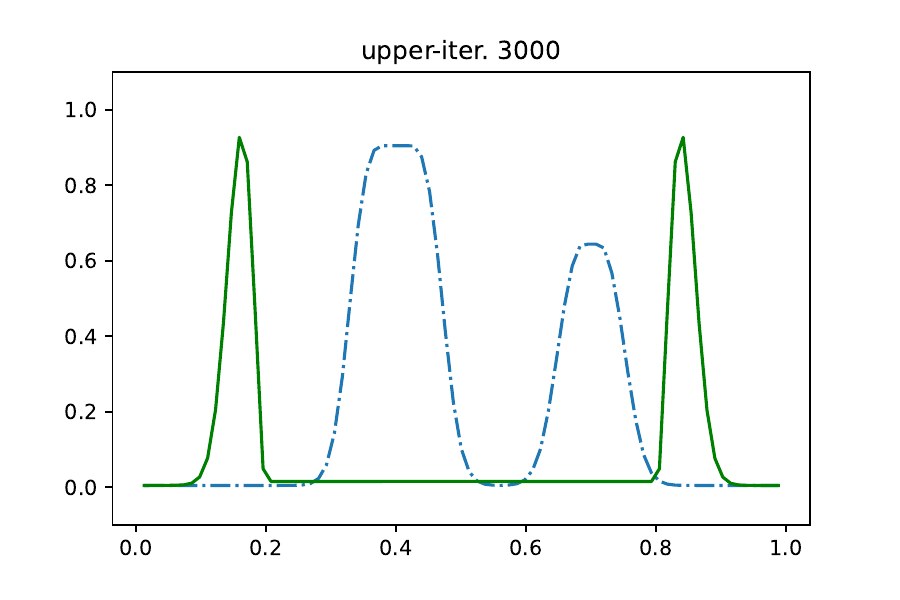}
\includegraphics[trim=2.8cm 0cm 2.5cm 0cm,clip=true,scale=0.35]{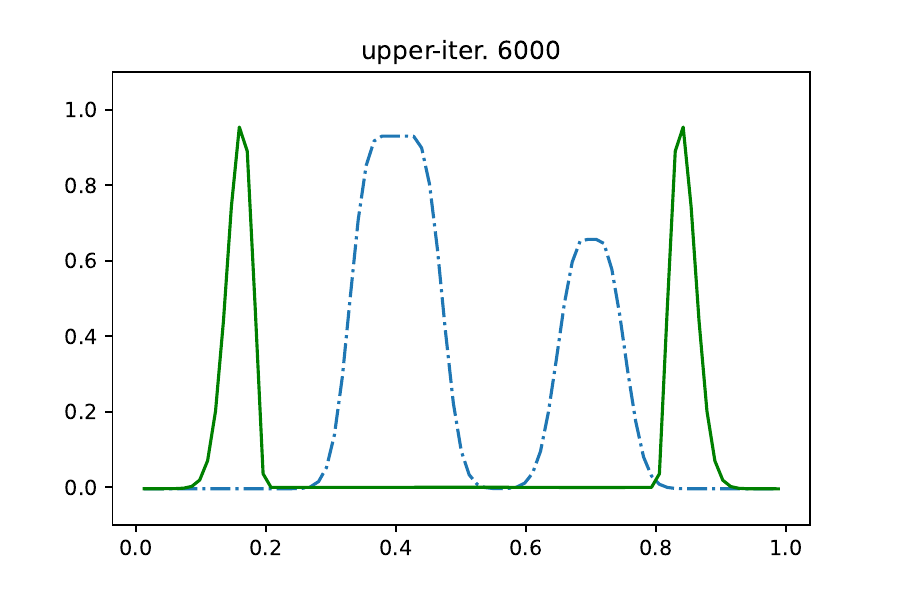}
\includegraphics[trim=2.8cm 0cm 2.5cm 0cm,clip=true,scale=0.35]{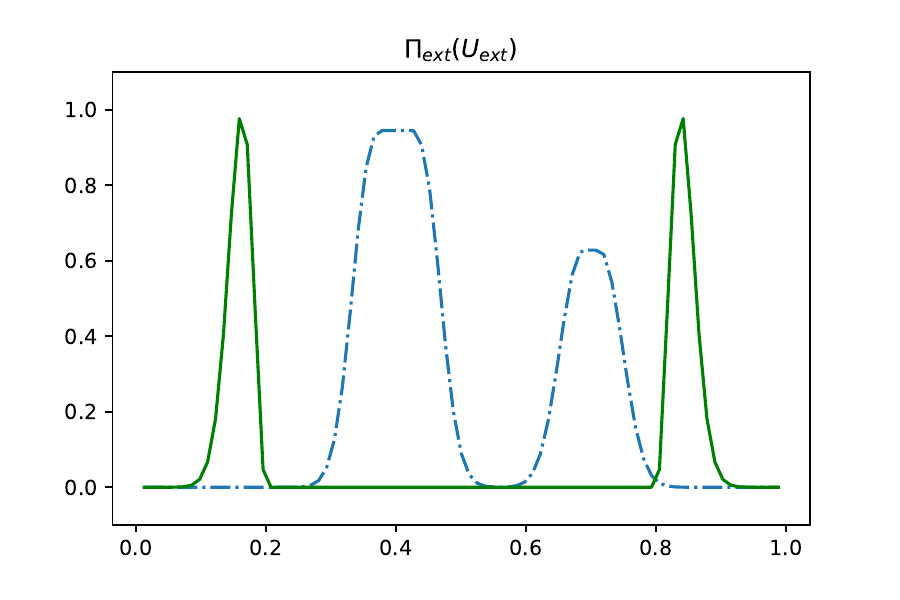}

\caption{Reaction applied to state trajectory. Plot two time moments: $t=0.01, 0.1$.}\label{fig-traject-f}
\end{figure}

\noindent\textbf{Lower-level trajectory. }In Fig.~\ref{fig-traject-U}, we demonstrate the connection between the sequential bi-level Algorithm  to the incremental load method, as remarked in Discussion \ref{dis:incremental}. To this end, we study the Fisher equation with terminal time data, collecting the output state from each lower-level. Starting from a zero initial guess (top-left panel), the intermediate state $u^j$ gradually evolves with $j$, and obtains the correct shape of the exact state (bottom-right) after 500 upper-iterations. Recall that each lower-level produces an approximate state associated to the incremental reaction $\Pi^j\neq\Pi^\dagger$. As the upper-iterate converges to $\Pi^\dagger$, these intermediate states are driven towards $u^\dagger=S(\Pi^\dagger)$, verifying the analysis in Section \ref{sec:lowe-trajectory}. In the same vein, Fig.~\ref{fig-traject-f} shows the evolution of $\Pi^j$  applied to snapshots of $u^j_{\kappa(j)}$.

\section{Outlook}\label{sec:outlook}

In this work, we have developed the sequential bi-level algorithm, an accelerated version of the original scheme introduced in \cite{nguyen24}. Several aspects have been explored: regularization effect, multi-scale effect, acceleration effect and incremental lower-level trajectories. Our proposed bi-level algorithm is designed for, but not limited to, inverse problems with nonlinear time-dependent PDEs. The application to reaction-diffusion systems demonstrates the universality of the bi-level schemes in a wide range of practical applications. Going forward, we intend to expand this study in the following directions:

\begin{itemize}
\item Thus far, our analysis has required the usage of Landweber regularization in both levels. Substituting Landweber for its Nestorov accelerated version \cite{Neubauer17,HubmerRamlau17} could lead to a dramatically improved acceleration effect in the sequential bi-level Algorithm \ref{algorithm3}; \blue{see Remark \ref{rem:generalization}}.
\item A key contribution of this article is the treatment of the stopping rules. One particularly interesting question is how to optimally balance the upper-stopping index $\jmax$ against the lower-stopping indices $\kappa(j)$, $j\leq \jmax$ in such a manner that the overall number $\sum_{i=1}^{\jmax}\kappa(j)$ of total iterations carried out by Algorithm \ref{algorithm3} is minimized. 

\item \blue{Obtaining convergence rates with respect to noise requires regularity of the ground truth, in the form of (variational) source conditions, as the problem is ill posed. Formally establishing these relationships for the bi-level method is an important question that merits further investigation.}

\item 
\blue{Another path towards optimality is the study of \emph{one-shot} methods \cite{Taasan91, MarcellaMarcellaVu}, where only a single iteration takes place in each lower-level step. This phenomenon has been partially observed in Figure 3 for the sequential version; if successful, this has the potential to provide extreme computational speedups.}

\item On the practical side, we plan to carry out a thorough numerical comparison of the three inversion approaches: reduced, all-at-once and bi-level, \blue{taking various observation strategies into account}. Target applications include hyper-elasticity, see Discussion \ref{dis:incremental}, and nonlinear waves in helioseismology.

\end{itemize}

The author also wishes to establish the tangential cone condition for infinite dimensional unknown especially with restricted measurements, being known so far an unsolved question.

\bmhead{Acknowledgements} The author indebted to Christian Aarset for thoroughly proofreading the manuscript. The author wishes to thank the reviewers for their insightful suggestions, leading to an improvement of the manuscript. The author acknowledges support from the DFG through Grant 432680300 - SFB 1456 (C04).

\bibliography{lit}

\end{document}